\documentclass{article}       % onecolumn (second format)
\usepackage{graphicx}
\usepackage[latin1]{inputenc}
\usepackage[english]{babel}
\usepackage{graphicx}
\usepackage{amsmath,amsfonts,amssymb,amsthm}
\usepackage{bbm,wasysym,stmaryrd}
\usepackage{subfigure}
\usepackage{color}
\usepackage{hyperref}

\newcommand{\R}{\mathbb{R}}

\newcommand{\derpart}[2]{ \frac{\partial #1}{\partial #2} } %% Partial derivative
\newcommand{\der}[2]{ \frac{\text{d} #1}{\text{d} #2} }  %% Partial derivative
\theoremstyle{plain}
\newtheorem{theorem}{Theorem}
\newtheorem{prop}{Proposition}
\newtheorem{lemma}{Lemma}
\newtheorem{rem}{Remark}

\newcommand{\eps}{\varepsilon}

\definecolor{orange}{rgb}{1.00,0.50,0.0}

\graphicspath{{./Figures/}}

\begin{document}

\title{Canard explosion in delayed equations with multiple timescales %\thanks{Grants or other notes
}
% \subtitle{Do you have a subtitle?\\ If so, write it here}


\author{Maciej Krupa\footnotemark[1]  \and Jonathan D. Touboul\footnotemark[1] \footnotemark[4]}
\renewcommand{\thefootnote}{\fnsymbol{footnote}}
\footnotetext[1]{MYCENAE Laboratory, Inria Paris-Rocquencourt, fistname.lastname@inria.fr}
\footnotetext[4]{The Mathematical Neurosciences Laboratory, Center for Interdisciplinary Research in Biology (CNRS UMR 7241, INSERM U1050, UPMC ED 158, MEMOLIFE PSL*)}

\renewcommand{\thefootnote}{\arabic{footnote}}
\date{\today}

\maketitle

\noindent \textbf{Abstract}	
We analyze canard explosions in delayed differential equations with a one-dimensional slow manifold. This study is applied to explore the dynamics of the van der Pol slow-fast system with delayed self-coupling. In the absence of delays, this system provides a canonical example of a canard explosion.  We show that as the delay is increased a family of `classical' canard explosions ends as a Bogdanov-Takens bifurcation occurs at the folds points of the S-shaped critical manifold.
\medskip


\setcounter{secnumdepth}{2}
\setcounter{tocdepth}{2}

\bigskip
\bigskip
\hrule
\tableofcontents
\hrule

\section*{Introduction}
Nonlinear dynamical systems with multiple timescales and delays are essential in applications. For instance, realistic models of neuronal dynamics accounting for the dynamics of neuronal areas involve several excitable elements, whose dynamics occur on very different timescales, interacting after delays due to the transmission of information through synapses. Similar problems arise in different domains, including mechanical systems~\cite{campbell-etal:09}, macroscopic phenomena arising in chemistry, physics or social science. Such nonlinear systems involving multiple timescale dynamics and delays generally display a rich phenomenology, and particularly a wide repertoire of complex periodic behaviors. Slow-fast systems have been chiefly analyzed in finite-dimensional contexts. The topic of the present paper is to analyze the role of delays in dynamics of slow-fast systems.

Slow-fast system have attracted a lot of attention from theoreticians and applied mathematicians. One phenomenon of particular interest in such systems is the so-called canard explosion, that describes a very fast transition, upon variation of a parameter, from a small amplitude limit cycle to a \emph{relaxation oscillation}, type of periodic solution consisting of long periods of quasi static behaviors interspersed with short periods of rapid transitions. These oscillations are ubiquitous in systems modeling chemical or biological phenomena~\cite{grasman:87,lagerstrom:88}. Canards were first studied about thirty years ago~\cite{benoit-etal:81} in the context the van der Pol (vdP) equation with constant forcing. The authors showed that close to a Hopf bifurcation in this system, a small change of the forcing parameter leads to such a fast transition from small amplitude limit cycles to large amplitude relaxation cycles. This canard explosion happens within an exponentially small range of the control parameter. These phenomena generically arise in two-dimensional dynamical systems~\cite{krupa-szmolyan:01}. In higher dimensional systems with multiple timescales, more complex oscillatory patterns may arise. Two examples are given by the so-called Mixed Mode Oscillations (MMO)~\cite{desroches:12} and bursting~\cite{rinzel-ermentrout:98}. Here, we show how to extend the theory of canard explosions to the setting of delayed equations. In a companion paper~\cite{krupa-touboul:14b}, we complete this  characterization by investigating theoretically and numerically the emergence of mixed-mode oscillations and bursting. 

The problem of canard explosions in delayed equations was first addressed by Campbell \emph{et al}~\cite{campbell-etal:09} for the analysis of a model of controlled drilling in the limit of small delays. In this regime, using the property that such systems present a two-dimensional inertial manifold~\cite{chicone:03}, they propose a two-dimensional ODE representation of the infinite-dimensional delayed system in the regime of small delays. This allows them to use the standard analysis of canards explosions in two dimensions and obtain a picture consistent with simulations of the original delayed system. Here, we do not restrict our analysis to small delays, and therefore reduction to a small dimensional ODE is no more possible. In this general case, canard explosions persist. Indeed, as shown in~\cite{krupa-szmolyan:01}, the generic mechanism of canard explosions in two dimensions relies upon Fenichel theory~\cite{fenichel:79}, the existence of connections in the fast subsystem, and the analysis of trajectories near non-hyperbolic points (fold points and canard points). For higher dimensional problem, center manifold reduction near fold points and canard points is necessary. Such elements are found in general delayed systems with one-dimensional slow manifold, as is the case of our system of interest, the self-coupled delayed van der Pol oscillator. In particular, Fenichel theorem for advance and delayed equations has been developed in~\cite{hupkes-standstede:10}, center manifold theory is classical~\cite{diekmann1995delay,faria-magalhaes:95,campbell2009calculating}. We show that these elements lead to the presence of canard explosions in delayed equations. 

We apply this theory to the delayed self-coupled vdP oscillator, and show the existence of canard explosions. In a companion paper~\cite{krupa-touboul:14b}, we show that beyond canard explosion, the system presents a richer dynamics than the sole canard explosion. Indeed, the fast dynamics is described by a one-dimensional delayed equation manifesting highly non-trivial dynamics and yielding, in the slow-fast system, complex oscillatory patterns including small cycles, relaxation cycles, MMOs, bursting and chaos. 

The paper is organized as follows. In section~\ref{sec:Theory} we present some general results on canard explosion for a class of delay differential equations. In section~\ref{sec:CanardvdP} we apply these results to the retarded van der Pol system. Appendix~\ref{append:fenichel} gives an overview of Fenichel's theory in the setting of our problem. Finally, appendix~\ref{append:Campbell} reviews the results of Campbell \emph{et al}~\cite{campbell-etal:09} in the case of small delays. 

% The paper is organized as follows. In section~\ref{sec:model}, we discuss neuroscience motivations for the analysis of canard explosions in delayed systems. Theoretical section~\ref{sec:Theory} gives extensions of the classical results on canard explosion~\cite{krupa-szmolyan:01,krupa2001extending} to delayed systems. These results are applied to the delayed vdP system in section~\ref{sec:vdP} and~\ref{sec:Complex}: section~\ref{sec:vdP} focuses on the oscillatory dynamics of the system and provides analysis of the full system and of the fast system, while section~\ref{sec:Complex} shows how these elements are the source of complex oscillatory patterns including MMOs, bursting and chaotic behaviors. 

\section{General Theory}\label{sec:Theory}
The theory of canard explosion developed for finite-dimensional ordinary differential equations, relies on three main ingredients: the existence and persistence of slow manifolds (Fenichel theory), a center manifold theory (generally near a fold point or a more degenerate point, e.g. Bogdanov-Takens point) and the existence and persistence of connections in the fast system. In this section, we show that all these elements persist in the infinite-dimensional setting of delayed differential equations, and that together, they yield canard transitions. 

\subsection{Fenichel theorem for delayed equation}
The Fenichel theorem \cite{fenichel:79} proves the existence of slow manifolds and their persistence in slow-fast dynamical systems. While the original result holds for finite-dimensional equations, much effort has been devoted to extend this theory to infinite dimensional systems. In particular, the theorem has recently been extended to the context of advanced-retarded
equations (more general than studied here) by Hupkes and Sandstede \cite{hupkes-standstede:10}. For such equations, the absence of a semiflow due to the advance term prevents from using classical analysis based on semi-groups~\cite{diekmann1995delay}. Their result of applies to the case of delayed differential equations. However, in that simpler case, we can use usual methods of spectral decomposition of the system and the semiflow of the equation, therefore readily extending the more classical techniques, developed in~\cite{diekmann1991center,diekmann1995delay}, in order to prove Fenichel theory. This is why we state here a result which is convenient for our purposes, and include a sketch of the proof based on these methods in appendix~\ref{append:fenichel}.

Let $X= C([-h, 0],\R^n)$. We consider the following delay equation:
\begin{align}\label{eq-delsf}
\begin{split}
x_t'=& \int_0^h d\zeta(y_t, \tau) x_{t-\tau} + f(x_t, y_t, \eps)\\
y_t'=&\eps g(x_t, y_t),
\end{split}
\end{align}
with $x\in\R^n$, $y\in\R$, $f\; :\;X\times \R\times \R_+\to\R^n$,
$f(0,y,0)=0$ and $D_1f(0,y,0)=0$. We denote by $x^t$ the element of $X$ corresponding to the map $x^t(\theta)=x_{t+\theta}$ for $\theta \in [-\tau,0]$. The above equation is classically written as a dynamical system in terms of the variable $x^t$ taking values in $X$.
The fast subsystem, obtained by setting $\eps=0$ in  \eqref{eq-delsf}, is given by
\begin{align}\label{eq-delf}
\begin{split}
x'=& \int_0^h d\zeta(y, \lambda, \tau) x_{t-\tau} + f(x_t, y,\lambda),
\end{split}
\end{align}
with $y$ playing the role of a parameter. The set of equilibria of \eqref{eq-delsf} parametrized by $y$ is known as the critical
manifold. We denote this set by $S$. Suppose that a segment of $S$ can be represented as a graph of a function $\phi\; :\; [y_1, y_2] \to \R^n$.
Then we can translate this segment of $S$ to the origin. Hence we can assume that  
$x=0$ is a solution of \eqref{eq-delf}, referred to as the \emph{trivial equilibrium}. Moreover, we can include the linear part of $f$ at $0$ in the term containing the integral.
As a result of this rearrangement  we have $f=O(|x|^2)$. 
The associated \emph{dispersion relationship} at the trivial equilibrium (obtained by linearization of~\eqref{eq-delsf} and evaluation on exponential functions with parameter $\lambda$) reads:
\begin{equation}\label{eq-defdelta}
\Delta(y,\lambda)=\lambda I-\int_0^h d\zeta(y, \tau)e^{-\lambda\tau}.
\end{equation}
Characteristic exponents governing the stability of the trivial solution are the values of $\lambda$ such that $\Delta(y,\lambda)=0$. 

Now that these elements have been introduced, we state a generalization of the Fenichel theorem \cite{fenichel:79} to the context of delay equations of the form
\eqref{eq-delsf}. As mentioned this result has been proved in much larger generality in \cite{hupkes-standstede:10}.
Our proof may be easy to follow due to the relatively simple setting and more `classical' approach. Moreover, the proof is used further 
in the paper as a basis for the proof of the theorem on canard explosion (more specifically, to prove the forthcoming Lemma~\ref{lem-exten}).

\begin{theorem}[Hupkes \& Standstede~\cite{hupkes-standstede:10}]\label{thm-fen} 
Suppose that there exist $y_1< y_2$ such that the characteristic roots for $y\in [y_1,y_2]$ (i.e., solutions of $\Delta(y,\lambda)=0$) are not on the imaginary axis. 
Then, for $\eps>0$ sufficiently small, there exist:
\renewcommand{\theenumi}{(\roman{enumi})}
\begin{enumerate}
	\item a slow manifold $S_\eps$ of the form $x=\phi(y)$, $y\in [y_1, y_2]$,
	satisfying $\phi=O(\eps)$,
	\item a finite dimensional unstable manifold $W^u_\eps$ consisting of all the solutions that are exponentially 
 repelled from $S_\eps$. Any solution starting close to $S_\eps$ becomes $O(e^{-c/\eps})$ close to $W^u_\eps$ before leaving a small neighborhood of $S_\eps$. 
\end{enumerate}
\end{theorem}

The proof of this theorem is provided in appendix~\ref{append:fenichel}. 

\subsection{Center manifold near singularities of the fast system}\label{sec:cmgen}
In this section we turn our attention to the behavior of the system close to singularities of the fast system; we discuss center manifold reductions around fold, 
canard and Bogdanov-Takens points. We consider a system of the form \eqref{eq-delsf}, and assume that there exists a fold point. This means that there exists 
$y_0$ such that the trivial solution of \eqref{eq-delsf} has a simple $0$ eigenvalue, i.e. the equation
$\Delta(y_0,\lambda)=0$ has a simple root $\lambda=0$. For simplicity of notation we assume that $y_0=0$.
Let $L$ denote the linear operator on $X$ given by
\[
L(\Phi) =\int_0^h   d\zeta(0, \tau) \Phi(-\tau)d\tau.
\]
It follows that $L$ has a simple $0$ eigenvalue.
We consider the extended system
\begin{align}\label{eq-delsfext}
\begin{split}
x'=& \int_0^h d\zeta(y, \tau) x_{t-\tau} + f(x_t, y, \eps)\\
y'=&\eps g(x, y)\\
\eps'=& 0.
\end{split}
\end{align} 
Note that the point $(0,0,0)$ is a non-hyperbolic equilibrium with three eigenvalues equal to $0$.
It follows from center manifold theory for delay equations ~\cite{faria-magalhaes:95,diekmann1991center} that there exists a three dimensional center manifold
containing this point. Let $\Phi$ be the eigenfunction of the $0$ eigenvalue and let $\Psi$ the 
the eigenfunction of the $0$ eigenvalue of a suitably chosen adjoint operator, which are both constant functions 
in the case of $\lambda=0$. Let $P$ be the projection
with ${\rm Im}\,(P)={\rm span}(\Phi)$ and the kernel given by the direct sum of the remaining eigenspaces.  
We have the following result.  
\begin{prop}\label{prop-cm}
There exists a function $h: \R^3 \to {\rm ker}\, P$ such the center manifold is given by
\[
\{(x_c\Phi+h(x_c,y,\eps),y,\eps)\; :\, (x_c,y,\eps) \mbox{ are in a small neighborhood of $(0,0,0)$.}\}
\]  
The reduction of  \eqref{eq-delsfext} to the center manifold has the form 
\begin{align}\label{eq-delsfred}
\begin{split}
x_c'=&f_c(x_c, y,\eps)\\
y'=&\eps g_c(x_c, y, \eps)\\
\eps'=&0,
\end{split}
\end{align}
with
\begin{align}\label{eq-cmfns}
\begin{split}
g_c(x_c,y,\eps)=&g\Big(x_c\Phi+{\mathbf r}h(x_c,y,\eps),y,\eps\Big),\\
f_{c}(x_c, y,z)=&\Psi{\mathbf r}(f\Big(x_c\Phi+h(x_c,y,\eps),y,\eps)\Big),
\end{split}
\end{align}
where ${\mathbf r}$ is the operator assigning to an element $x\in X$ its value at $0$
$x(0)$.
\end{prop}
\begin{proof}
	We refer to ~\cite{faria-magalhaes:95} for details
	on center manifold reduction. Here we just point out that the eigenspace
	of $0$ consists of a vector $\phi\in X$ and the vectors $(0,1,0)^T$ and $(0,0,1)^T$
	in the $y$ and $\eps$ directions. Hence, the reduction in the $y$ and $\eps$ directions
	is the same as in the ODE case.	
\end{proof}
We say that the point $(0, 0, 0)$ is a non-degenerate fold point if $f_{c,xx}(0,0,0)\neq 0$,
$f_{c,y}(0,0,0)\neq 0$ and $g_c(0,0,0)\neq 0$. If $g_c(0,0,0)=0$ then $(0,0,0)$ is a canard
point. Non-degeneracy conditions for a canard point are complicated in general. These are recalled in 
the course of appendix~\ref{append:Campbell}. We will
not restate them here but rather refer the reader to \cite{krupa2001extending} (the presence 
of delays does not modify these conditions since we have reduced the problem on a finite-
dimensional manifolds). In the specific example we 
consider in the sequel these conditions are simpler and will be verified. 
The dynamics of a fold point or a canard point restricted to the center manifold  is
now as described in  \cite{krupa2001extending}.

Finally,
consider a system of the form \eqref{eq-delsf} at a fold point with an additional degeneracy of Bogdanov-Takens type. This means that there exists 
$y_0$ such that the trivial solution of \eqref{eq-delsf} has a double $0$ eigenvalue with one eigenvector and one generalized eigenvector. 
In the extended system~\eqref{eq-delsfext}, this point is thus a non-hyperbolic equilibrium with four eigenvalues equal to $0$. It follows from center manifold theory for delay equations~\cite{faria-magalhaes:95,campbell2008zero} that there exists a four dimensional center manifold
containing this point. Reduction of the system around this point is similar to that of the previous section, and detailed calculations are provided in the particular case of the delayed van der Pol system in~\cite{krupa-touboul:14b}. 

\subsection{Canard explosion}
We consider a system of the form \eqref{eq-delsfext}, depending on a regular parameter $\mu$ with an S-shaped critical manifold, as shown in Fig. \ref{fig-S}.
Fix $\mu=\mu_0$. Let
\[
S=S_-\cup \{(x_m, y_m)\}\cup S_r\cup \{(x_M, y_M)\}\cup S_+.
\]
The following hypothesis are necessary for a canard explosion result:\\[0.2ex]
 {\bf (H1)} $(x_m, y_m)$ and $(x_M, y_M)$ are a non-degenerate canard point and a non-degenerate fold point. 
\\[0.2ex]
{\bf (H2)} $S_-\cup S_+$ consists of sinks of \eqref{eq-delf} and $S_r$ consists of saddle points with one dimensional
unstable direction. \\[0.2ex]
{\bf (H3)} There exist connections from $S_r\to S_\pm$ as shown in Fig.  \ref{fig-S}.\\[0.2ex]
We will verify hypotheses {\bf (H1)} and {\bf (H2)} in the context of  \eqref{eq:vdpdelay}
and present numerical evidence that {\bf (H3)} is also satisfied.
For the remainder of this section we assume that {\bf (H1)-(H3)} hold.
\begin{figure}[htbp]
	\centering
	      \includegraphics[width=.5\textwidth]{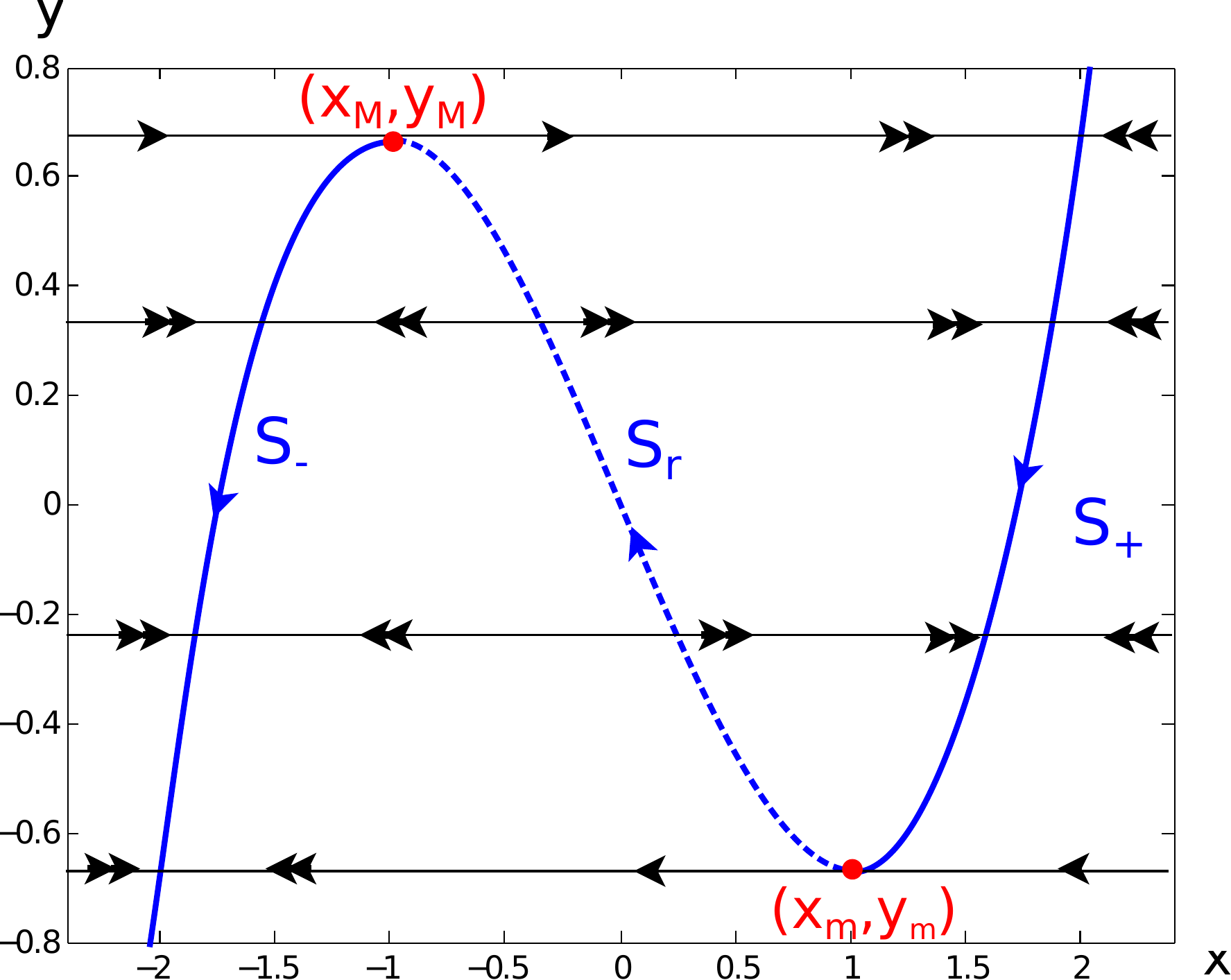}
	\caption{Slow and fast dynamics for $\eps=0$ necessary for a canard explosion. Thick black $S$-shaped curve is the critical manifold, thin phaselines are schematics of the fast dynamics. Picture drawn in the case of the vdP system~\eqref{eq:vdpdelay}.}
	\label{fig-S}
\end{figure}
We begin with a result on extending a center manifold $C_\eps$ existing near the canard point.
\begin{lemma}\label{lem-exten}
The manifold $W^u(S_{r,\eps})$  obtained by Theorem \ref{thm-fen} and  the center manifold $C_\eps$ 
near the canard point obtained by Proposition \ref{prop-cm} can be chosen to overlap. More specifically, there exists a choice of $W^u(S_{r,\eps})$ (including a choice of 
$S_{r,\eps}$ itself)  and a choice of  a center manifold $C_\eps$ so that the two manifolds  overlap on a neighborhood of a segment of $S_{r,\eps}$.  
\end{lemma}
This lemma is illustrated in Fig.~\ref{fig:Manifold}.
\begin{figure}[htbp]
	\centering
	          \includegraphics[width=.4\textwidth]{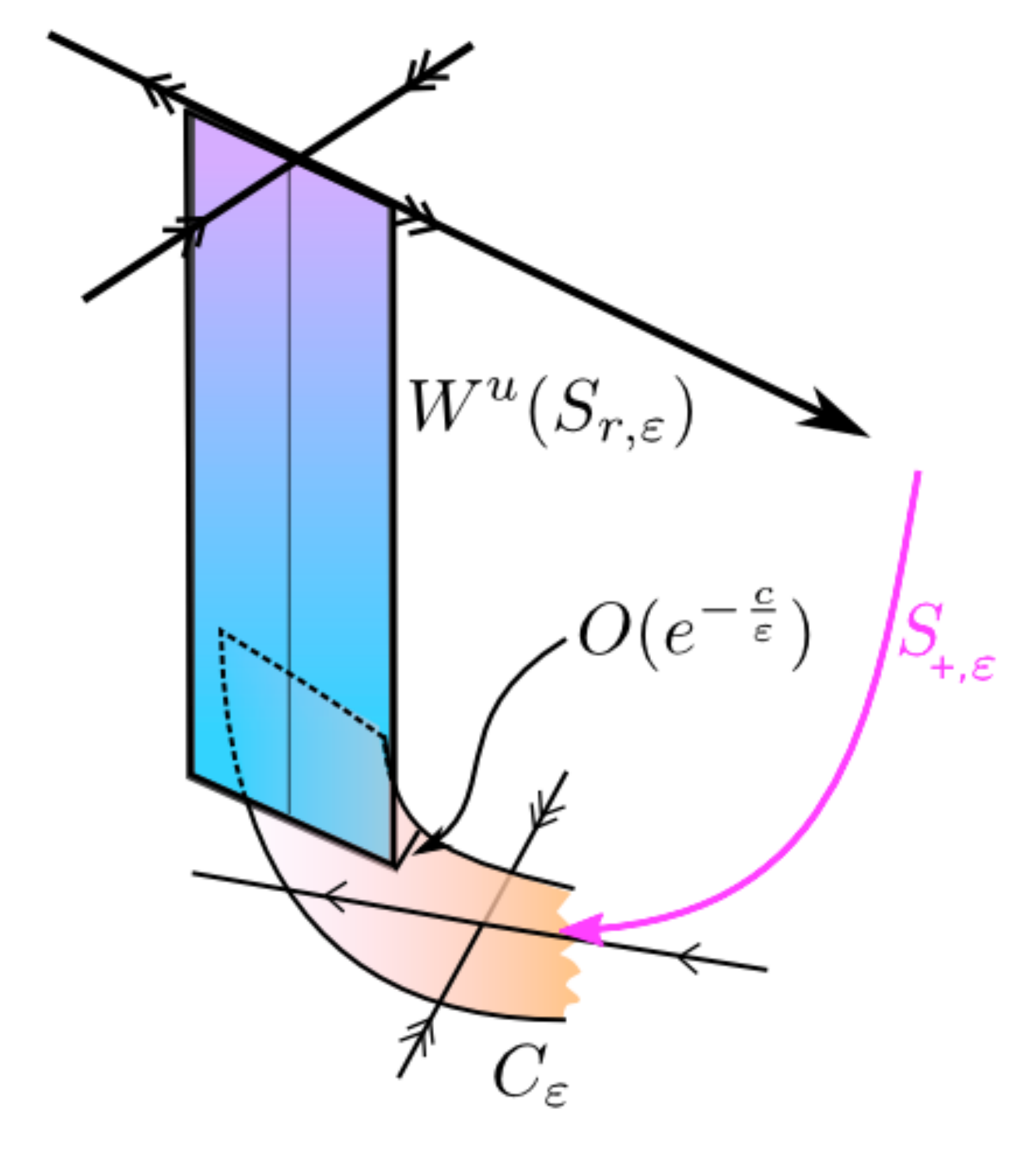}
	\caption{Center manifold $C_{\eps}$, stable manifold $S_{+,\eps}$ and the unstable manifold of the saddle-type slow manifold $S_{r,\eps}$ (see Lemma~\ref{lem-exten}).}
	\label{fig:Manifold}
\end{figure}
\begin{proof}First note that the manifolds $C_\eps$ and $W^u(S_{r,\eps})$ can be chosen so that their regions of existence overlap
(see Fig.~\ref{fig:Manifold}).
On the overlap, the tangent spaces to $C_\eps$  and $W^u(S_{r,\eps})$ are close to each other, by construction.
Hence, on the overlap, the two manifolds are close to each other. Note further that near the canard
point the stable part of the spectrum is bounded away from $0$, and hence, one can choose a neighborhood 
$V\subset X$ of $W^u(S_{r,\eps})$ such that the trajectories in $V$ are exponentially attracted to $W^u(S_{r,\eps})$ as long as
they stay in $V$.  By choosing the domain of existence of $W^u(S_{r,\eps})$ so that it extends sufficiently close to the  fold we can guarantee
that $C_\eps$ has a non-empty intersection with  $V$. More specifically, we can choose a subset of $C_\eps$ containing an interval 
$I_0$ defined by $y=y_0$ bounded by two points $(x_{c1}, y_0)$ and  $(x_{c2}, y_0)$ such that $\Phi_{c,t}(x_{c1}, y_0)$ escapes
from $C_\eps$ towards $S_{-,\eps}$ and  $\Phi_{c,t}(x_{c2}, y_0)$ escapes from $C$ towards $S_{+,\eps}$ ($\Phi_{c,t}$ denotes the flow of \eqref{eq-delsfred}).
We consider an interval $I_1\subset \{y=y_1\}\subset C_\eps$ given by a transition map from $I_0$ to $I_1$ by the flow $\Phi_{c,t}$
By exponential attraction of $W^u(S_{r,\eps})$ the interval $I_1$ is exponentially close to $W^u(S_{r,\eps})$.
We now extend the manifold $C_\eps$ by applying the semi-flow $\Phi_{t}$ to initial conditions in $I_1$ and intersecting with $V$.
This gives a $C^0$ manifold $\tilde C_\eps$
exponentially close to $W^u(S_{r,\eps})$. To see that 
there exists a choice of $C_\eps$ such that this extension is smooth
we modify the construction of $C_\eps$ and $W^u(S_{r,\eps})$ by first multiplying $g$ in \eqref{eq-delsflin}
by a cut-off function which is $0$ on a small neighborhood of the canard point and $1$ outside of a small neighborhood of the canard point.
Subsequently we apply the argument sketched in Appendix \ref{append:fenichel} with the domain in the slow direction of $y$ extended
to include a neighborhood of the canard point and with the linear flow defined by the modified  \eqref{eq-delsflin}.
This flow is stationary in the $y$ direction as long as the modified $g$ equals $0$ and is defined by \eqref{eq-lin1}
when the modified $g$ is positive. Note that the proof now yields a manifold which is gives a choice of $C_\eps$ near the canard point
and a choice of $W^u(S_{r,\eps})$ away from the canard point. The proof can be extended in the standard fashion to show that 
the manifold obtained in this manner is smooth. \end{proof}
%\vspace{-0.cm}
\begin{lemma}\label{lem-inclu}
There exists a choice of the stable slow manifold $S_{+,\eps}$ and the center manifold $C_\eps$ such that a segment  
of $S_{+,\eps}$ is included in $C_\eps$.  
Moreover, $C_\eps$ can be chosen as specified in Lemma \ref{lem-exten} and there exists a smooth curve in the parameter space of the
form $(\mu_c(\eps),\eps)$ such that if $\mu=\mu_c(\eps)$ then  $S_{+,\eps}$ connects to $S_{r,\eps}$. 
The connection from $S_{+,\eps}$ to $S_{r,\eps}$ is called a canard solution.
\end{lemma}
\begin{proof}We first  modify the construction of $C_\eps$ and $S_{+,\eps}$ to ensure that a segment of $S_{+,\eps}$
is included in $C_\eps$. 
Note that $S_{+,\eps}$ is defined as a graph of a function $\Phi_\eps\; :\; \R\to X$. We define 
$\varphi_\eps={\bf r}(\Phi_\eps(y))$, where ${\bf r}$ is the restriction operator introduced in Appendix \ref{append:fenichel}.
The invariance of $S_{+,\eps}$ now implies
\begin{equation}\label{eq-invcond}
\eps\varphi'_\eps(y)g(\varphi_\eps(y), y)=\int_0^h d\zeta(y, \tau) \Phi_\eps(y)(-\tau)+F(\varphi_\eps(y),y).
\end{equation}
Note that $\varphi_\eps$ is not defined on the neighborhood of $(x_m, y_m)$. We extend it by an arbitrary function, just making
sure the extension has the same degree of regularity.
We fix $y_1$ and $y_2$ satisfying $y_m<y_1<y_2<y_M$ and let $\kappa:\R\to\R$ be a non-negative function equal to $1$ on
$[y_1, y_2]$ and $0$ on a neighborhood of $y_m$. Let $\psi_\eps(y)=\kappa(y)\varphi_\eps(y)$.
We define a new variable 
\begin{equation}\label{eq-defxtil}
\tilde x=x-\psi_\eps(y)
\end{equation} 
and transform \eqref{eq-delsf} to the new variables. It follows from \eqref{eq-invcond} that \eqref{eq-delsf} transforms to
\begin{align}\label{eq-delsft}
\begin{split}
x'=& \int_0^h d\zeta(y, \tau) x_{t-\tau} + \tilde f(x_t, y, \eps)\\
y'=&\eps g(x+\varphi_\eps(y), y),
\end{split}
\end{align}
where $\tilde f(0,y,\eps)=D_x\tilde f(0,y,\eps)=0$. We now pick $y_3$ between $y_2$ and $y_1$
and let $\eta(y)$ be a $C^\infty$ function which satisfies $\eta'>0$ for $y> y_3$, $\eta(y_2)=0$,
$\eta(y)=1$ for $y<y_3$. We consider the system
\begin{align}\label{eq-delsftm}
\begin{split}
x'=& \int_0^h d\zeta(y, \tau) x_{t-\tau} + \tilde f(x_t, y, \eps)\\
y'=&\eps g(x, y)\eta(y).
\end{split}
\end{align}
Note that \eqref{eq-delsftm} has a saddle type equilibrium point at $(0, y_2)$ with one dimensional unstable manifold.
We now construct a center manifold for \eqref{eq-delsft} near the canard point $(x_m,y_m)$.
Note that, by choosing $y_2$ sufficiently small, we can ensure that the added saddle point
is on the center manifold $\tilde C_\eps$, as well as its unstable manifold, which, for $y\in (y_1, y_2)$, coincides
with the line $x=0$. To complete the proof of the claim we make two observations. First, since the dynamics of \eqref{eq-delsft}
and \eqref{eq-delsftm} coincide on a small neighborhood of the canard point, $\tilde C_\eps$ defines also a center
manifold of \eqref{eq-delsf}. Second, since the dynamics of \eqref{eq-delsftm} and \eqref{eq-delsft} are the same for $(y_2, y_3)$,
the line segment $\{(0,y), y\in (y_2, y_3)\}$ is both on $\tilde C_\eps$ and corresponds to a segment of $S_{+,\eps}$. The first
claim of the lemma follows.

To prove the second claim note that the argument in the proof of Lemma \ref{lem-exten} can be applied independently of the one described in
the preceding paragraph, so that the manifold $C_\eps$ can be extended all the way to the vicinity of $(x_M, y_M)$.
The existence of a connecting orbit from $S_{+,\eps}$ to $S_{r,\eps}$ is then a direct conclusion of the arguments in
\cite{krupa2001extending} as segments of both $S_{+,\eps}$ and $S_{r,\eps}$ are contained in
$\tilde C_\eps$. As in~\cite{krupa2001extending} we set up a Melnikov integral and observe that its value is determined, up to exponentially small terms, 
by the restriction of the flow on $C_\eps$ to a small neighborhood of the canard point. 
\end{proof}

\noindent  We now formulate conditions that guarantee the stability of canard cycles, see \cite{krupa-szmolyan:01}.
Let $\lambda_p(y)$ be the unique positive root of $\Delta(\lambda,y)=0$ corresponding to the 
saddle-type part of $S_0$. Let $\lambda_{n,+}(y)$ be the largest characteristic root corresponding  to the right branch
of $S_{+}$. We parametrize the branches of $S$ between the fold points, associated with $y=y_m$
and $y=y_M$ as $(y,\phi_-(y))$, $(y,\phi_{r}(y))$
and $(y,\phi_+(y))$, with $(y,\phi_-(y))$ and $(y,\phi_+(y))$ corresponding to the stable branches
and $(y,\phi_{r}(y))$ to the saddle type branch. 
 For every $y_*$ satisfying $y_m< y_* < y_M$ let
 \begin{align}\label{eq-defR}
 \begin{split}
 R_{n,+}(y)&=\int_{y_m}^{y_*} \frac{\lambda_{n,+}(y)}{g(\phi_+(y),y)}dy\\
 R_{n,-}(y)&=\int_{y_*}^{y_M} \frac{\lambda_{n,-}(y)}{g(\phi_-(y),y)}dy\\
 R_p(y)&=\int_{y_m}^{y_*} \frac{\lambda_p(y)}{g(\phi_{r}(y),y)}dy.
 \end{split}
 \end{align}
We make the following assumption:\\[0.2ex]
{\bf (H4)} $R_{n,+}(y_*) >R_p(y_*)$ for every $y_*$ satisfying $y_m< y_* < y_M$.
\noindent \begin{theorem}\label{th-canoh}
Suppose {\bf (H4)} holds, in addition to {\bf (H1)-(H3)}. Then, for every $\eps$ sufficiently small, there exists a family of canard cycles continuing
from small, Hopf type cycles to relaxation cycles, through canards with no head and 
subsequently canards with head. The transition from small canards
to canards with large head takes place in an exponentially small interval of the parameter $\mu$.
The cycles are stable and unique (at most one for each choice of $(\mu,\eps)$) and depend smoothly on the parameters.
\end{theorem}
\begin{proof}The hypotheses {\bf (H1)} - {\bf (H4)} guarantee the local part of canard explosion restricted
to the center manifold. By Lemmas~\ref{lem-exten} and~\ref{lem-inclu} we can choose $C_\eps$
so that it contains a segments of $S_{+,\eps}$ and extends to $W^u(S_r)$.
We can now measure the separation between $S_{+,\eps}$ and $S_{r,\eps}$
in the 2-dimensional center manifold $C_\eps$, in which the flow is as described in~\cite{krupa-szmolyan:01,krupa2001extending}.
When $S_{+,\eps}$ and $S_{r,\eps}$ are exponentially close, which corresponds to the parameter region
very close to the locus of a connection from $S_{+,\eps}$ to $S_{r,\eps}$, 
then the forward continuation of  $S_{+,\eps}$
follows $W^u(S_{r,\eps})$ for a time of order $O(1/\eps)$ and splits off  towards either $S_{+,\eps}$ or $S_{-,\eps}$.
Either way, it ends up being attracted to $S_{+,\eps}$ and returning very close
to itself in the vicinity of the canard point. We now define a section
of the flow $\Delta$ by the requirement $x_c=x_{m}$ and let $p_{+,\eps}= S_{+,\eps}\cap\Delta$. Further we consider a small neighborhood $U$
of $p_{+,\eps}$ in $\Delta$. Note that under the assumptions made above the Poincar\'{e} map from $U$ to $\Delta$  is well defined, provided that $U$ 
is sufficiently small. Due to {\bf (H4)} this Poincar\'{e} map is an exponential contraction mapping $U$ into itself,
which implies the existence of an asymptotically stable canard cycle. 
One can now apply standard theory for limit cycles to conclude that
canard cycles depend smoothly on parameters if $\eps>0$.
Moreover, all canard cycles must be in an exponentially small wedge
of the parameter plane around the curve $\mu_c(\eps)$ corresponding to a connection from
$S_{+,\eps}$ to $S_{r,\eps}$. By Lemma \ref{lem-inclu} $\mu_c(\eps)$
depends smoothly  on $\eps$ and other parameters, also in the limit $\eps\to 0$. 
\end{proof}

We define $p_{\rm can}\in\Delta$ the unique intersection point of $\Delta$ with a canard cycle
(whenever such an intersection exists). We now characterize how $p_{\rm can}$ and its derivatives with respect to the regular parameters and $\eps$
behave as $\eps\to 0$.

\begin{prop}\label{prop-smep0}
The point $p_{\rm can}$ is uniformly $O(e^{-c/\eps})$ close to $p_{+,\eps}$, where $c>0$ is a fixed constant.
Similarly any derivatives of $p_{\rm can}$ with respect to $\eps$ and regular parameters are
uniformly  $O(e^{-c/\eps})$ close to the corresponding derivatives of $p_{+,\eps}$.  An analogous statement holds for higher order
derivatives.
\end{prop}
\begin{proof}
We treat the case of canards without head, the other case is similar. Our results developed to this point
give smooth dependence of $p_{+,\eps}$ on regular parameters and $\eps$ in the limit $\eps\to 0$.
We claim that this assertion holds for $\Pi^k(p_{+,\eps})$, for any positive $k$, and the derivatives of $\Pi^k(p_{+,\eps})$
are exponentially close to the derivatives of $p_{+,\eps}$. An analogous statement about $p_{\rm can}$ now
follows from the following facts: $\Pi^k(p_{+,\eps})$ converges to $p_{\rm can}$ and $\Pi$ is a uniform contraction.
A similar argument can be applied to higher derivatives..

To see that the claim holds first note that the trajectory of $p_{+,\eps}$ is contained in $W^u(S_{r,\eps})$,
which is a 2D smooth manifold. Smoothness is assured to some point $\tilde p_{+,\eps}$, where the
trajectory leaves the vicinity of $S_{r,\eps}$. Now it is possible to adapt the methods developed above
to prove that the forward trajectory of $\tilde p_{+,\eps}$ gives a good choice of $S_{+,\eps}$.
To do this we can first modify the flow near $S_{r,\eps}$ so that $\tilde p_{+,\eps}$ is in a one dimensionnal 
unstable manifold of a true equilibrium on $S_{r,\eps}$ and the forward trajectory of $\tilde p_{+,\eps}$
is unchanged. We can now extend this unstable manifold, using a similar approach as in the proof
of the existence of a slow manifold, using the contraction near the equilibrium to prove smoothness.
Finally, we can extend this construction to the vicinity of $S_+$ using a variant of the trick used in 
Lemma \ref{lem-exten}. We use a smooth partition of unity (depending on $t$) to define a linear
operator which during the passage from $\tilde p_{+,\eps}$ is given by the linearization
along the unstable manifold and near $S_{a}$ is given by \eqref{eq-lin1}. It now follows that this trajectory gives a choice
of $S_{+,\eps}$ and the claim on smoothness follows for $k=1$. Now we apply a similar argument as in Lemma \ref{lem-inclu}
finding a smooth center manifold $C_\eps$ and a smooth manifold $W^u(S_{r,\eps}$ which contain the extension of  $S{+,\eps}$.
It follows that $\Pi(p_{+,\eps})$ is on the modified manifold $C_\eps$, so that we can apply the same argument to prove an analogous
result for $k=2$. Proceeding by induction we obtain the result for all $k$. The result on the estimate of the derivatives of $\Pi^k(p_{+,\eps})$
follows from the fact that each such point corresponds to a choice of $S_{+,\eps}$ and the derivatives of $S_{+,\eps}$ obtained by different constructions
cary by an exponentially small amount.
\end{proof}
%\begin{rem}\label{rem-morsm}
%To prove smoothness in $\eps$ in the limit $\eps\to 0$ we can  set up an operator on a space of loops,
%(more specifically a Liapunov-Schmidt operator \cite{golubitsky-langford:81}), which would have as fixed points the canard cycles we have found.
%The Liapunov-Schmidt operator can be expressed using a variation of constants formula similar to \eqref{eq-opsf},
%with the linear flow given by the Floquet flow along the canard. The integration  time in the term containing the integral
%would be the period of the canard cycle. Smooth dependence of canard cycles on parameters and smoothness of 
%curves in the parameter space corresponding to families of  canard cycles defined by a specific feature (like the maximal height)
%may be proved using a similar approach as in the proof of the center manifold theorem 
%\cite{diekmann1995delay}, i.e. by extending the operator to a larger space including the derivative terms and proving the existence
%of a fixed point in that space.
%\end{rem}
If (H4) does not hold it is possible to obtain partial results, based on the following.
\begin{theorem}\label{th-canloc}
Suppose that {\bf (H1)-(H3)} holds and $R_{n,+}(y_*) >R_p(y_*)$ for some $y_*\in (y_m, y_M)$.
Then there exists a smooth curve in the parameter plane of the form $(\mu(\eps),\eps)$
corresponding to the locus of existence of canard cycles with no head passing through the point $(x_m, y_*)$.
The canard cycles belonging to this family are asymptotically stable and 
depend smoothly on $\eps$. Similarly, suppose that $R_{n,-}(y_*)+R_{n,+}(y_M) >R_p(y_*)$ for some $y_*\in (y_m, y_M)$.
Then there exists a smooth curve in the parameter plane of the form $(\mu(\eps),\eps)$
corresponding to the locus of existence of canard cycles with head passing through the point $(x_M, y_*)$.
The canard cycles belonging to this family are asymptotically stable and 
depend smoothly on $\eps$.
\end{theorem}
\begin{proof}The proof is analogous to the proof of Theorem \ref{th-canoh}. 
	\end{proof}
\begin{rem}\label{rem-width}
Note that the condition $R_{n,-}(y_*)+R_{n,+}(y_M) >R_p(y_*)$ must be satisfied 
for $y_*$ sufficiently small. This implies that canards with `large head' must exist and be stable.
Similarly, canard explosion is locally always either subcritical or supercritical
if a non degeneracy condition holds (see \cite{krupa-szmolyan:01}). Finally, all canard cycles whose existence
follows from Theorem \ref{th-canloc}
are exponentially close in the parameter space to a segment of the canard solution
and therefore the $\mu(\eps)$ values they correspond to must be exponentially close to
$\mu_c(\eps)$. This means that there exists a weak version of a  canard explosion even if {\bf (H4)} does not
hold, namely a transition from small cycles to canard cycles with `large head' which occurs in an exponentially small
region.
\end{rem}

\section{Application to the delayed van der Pol system}\label{sec:CanardvdP}
We now investigate the presence of canard explosions in the delayed van der Pol (vdP) equation. This model, motivated by the analysis of firings of neurons, is given by the equations (see~\cite{krupa-touboul:14b}):
\begin{equation}\label{eq:vdpdelay}
	\begin{cases}
			x'_t=x_t-\frac{x_t^3}{3}+y_t+ J (x_t-x_{t-\tau})\\
			y_t'=\varepsilon (a-\,x_t).
	\end{cases}
\end{equation}
In that model, modification of the mean-field limit of a Fitzhugh-Nagumo system, $x$ represents the voltage of a cell and $y$ is a slow adaptation variable, $J$ represents the average coupling strength between neurons, $\tau$ the average delay of communication between the cells and $a$ is related to the input received by the neurons. When $\tau=0$, one recovers the classical vdP system. This equation has been the first example of system with canard explosion~\cite{benoit-etal:81}: for $a>1$, the system has a unique, globally attractive fixed point, that looses stability at $a=1$ through a supercritical Hopf bifurcation. A family of limit cycles emerges starting with small amplitude Hopf like periodic orbits that rapidly turn into relaxation cycles as $a$ is decreased, through canards with no head and then canards with head. 

We shall consider now the effect of the delay on the canard transition. We show that, as an application domain theorems~\ref{th-canoh} and~\ref{th-canloc}, this system displays canard explosion as delays are varied, for $J\tau<1$. For $J\tau >1$, the instability created by the delay interacts with the instability given by canard explosion leading to much more complex dynamics~\cite{krupa-touboul:14b}.

\subsection{Equilibria and stability of the fast equation}\label{sec:EquilibriaFast}
The fast equation is given by the solution of the singular limit $\varepsilon\to 0$ in equation~\eqref{eq:vdpdelay}, i.e.:
\begin{equation}\label{eq:FastDelayed}
			x'_t=x_t-\frac{x_t^3}{3}+y+ J (x_t-x_{t-\tau})
\end{equation}
where $y$ is considered as a parameter, corresponding to the value of $y_t$, which is constant in the singular limit. Fixed points are given by the solutions to the algebraic equation:
\[x-\frac{x^3}{3}+y=0,\]
which has three real solutions (fixed points) when $\vert y \vert < \frac 2 3$ and one fixed point otherwise. Note that the solutions to this equation constitute the critical manifold of~\eqref{eq:vdpdelay}. Fixed points can be written in closed from using Cardano's method. For $\vert y\vert >2/3$, the unique solution is given by:
\[x_0=\left(\frac{3y+\sqrt{9 y^2 -4}}{2}\right)^{1/3}+\left(\frac{3y-\sqrt{9 y^2 -4}}{2}\right)^{1/3}\]
and for $\vert y \vert < 2/3$, the three solutions are given by
\[x_k=2\cos\left(\frac 1 3 \arccos\left(\frac{3y}{2}\right) + \frac{2k \pi}{3}\right) \qquad, \qquad k=0,1,2.\]
and for $ y=\pm2/3$, there is a double root $x=\mp 1$ and a simple root $x=\pm 2$. There are hence three branches of fixed points: $x_+(y)\geq 1$ corresponding to the branch of solutions for $y\geq -2/3$, $x_-(y)\leq -1$ corresponding to $y\leq 2/3$ and $x_0(y) \in [-1,1]$ defined for $y\in[-2/3,2/3]$. This manifold is displayed in Fig.~\ref{fig-S}. 

We now show that $x_0(y)$ is a saddle with one unstable direction, and $x_{\pm}(y)$ are stable as long as $J\tau <1$.  In order to prove this, we analyze the characteristic roots $\xi$ of the system, i.e. solutions to~\eqref{eq-defdelta}, reading in our case:
\begin{equation}\label{eq:CharactRootsFast}
	\xi=1-(x^*)^2 +J -Je^{-\xi\tau}
\end{equation}
This equation may be solved using special functions\footnote{Indeed, the solutions to the characteristic equation are given by the Lambert functions $W_k$ (the different branches of the inverse of $x\mapsto xe^x$, see e.g.~\cite{corless:96}):
\[\xi = A+\frac 1 \tau W_k\left(-\tau J e^{-\tau A}\right)\]
with $A=1-(x^*)^2+J$. The stability of $x^*$ is hence governed by the sign of the real part of the rightmost eigenvalue, given by the real branch $W_0$ of the Lambert function, and if the argument of the Lambert function has a real part greater than $-e^{-1}$ the root is unique. If not, we have two eigenvalues with the same real part corresponding to $k=0$ or $-1$.}. Since for $\tau=0$, $x_0$ is a saddle with a single unstable direction and $x_{\pm}(y)$ are stable, we only need to show that there is no delay-induced bifurcation for $J\tau<1$. First, saddle node bifurcation arise when there exist characteristic roots $\xi$ equal to $0$: this occurs if and only if $x^*=\pm 1$, i.e. $y=\pm \frac 2 3$, independently of the delay $\tau$ and the coupling strength $J$. 

Hopf bifurcations in the fast system occur when $\xi=\mathbf{i}\zeta$ for $\zeta>0$. In that case, taking the imaginary part of the dispersion relationship, we obtain:
\[\zeta=J\sin(\zeta\tau), \qquad \text{i.e.} \qquad J\tau = \frac{\zeta \tau}{\sin(\zeta \tau)}\geq 1.\]
Therefore, the fast system does not undergoes Hopf bifurcations as long as $J\tau \leq 1$. 

We conclude that the stability of the branches of the critical manifold is as for $\tau=0$ as long as $J\tau <1$, ensuring validity of \textbf{(H1)} and \textbf{(H2)}. 
\subsection{Local canard explosion}
Due to center manifold reduction (proposition~\ref{prop-cm}), the analysis of local canard explosion is similar as for a system in two dimensions as performed in~\cite{krupa-szmolyan:01}. The existence of a global
canard explosion relies on the hypotheses {\bf (H3)-(H4)} that will be checked in the next sections, and the present section is devoted to the local canard explosion.

It follows from the calculations in Section \ref{sec:EquilibriaFast} that the line line segment $a=1$, $0\le \tau<1/J$ in the parameter space $(a,\tau)$ corresponds to the locus of canard points. To understand the details of local canard explosion we derive the reduction of \eqref{eq:vdpdelay} to a center manifold at the canard point,
see Section \ref{sec:cmgen} for a general description of such a reduction. To carry out the reduction we use the fact that a canard point is a special case of a 
a Bogdanov-Takens point. Further we observe that  \eqref{eq:vdpdelay} has the same structure as (3.8) in \cite{campbellJDDE}, hence we obtain our normal
form by following closely the approach of \cite{campbellJDDE}. In addition we take advantage of the fact that 
our nonlinearity is independent of the delay (see~\cite[Appendix]{krupa-touboul:14b} for a similar
reduction, in the context of Hopf bifurcation in \eqref{eq:vdpdelay}). The reduction followed by a reflection in the $\tilde x$ variable, yields the following system on the center manifold:
\begin{equation}\label{eq-syscm}
\begin{cases}
	\der{\tilde{x}}{\theta} =  -\tilde{y}+\tilde{x}^2-\frac{\tilde{x}^3}{3}+a_1\eps\tilde x+ {\rm hot}\\
	\der{\tilde{y}}{\theta} = \tilde{\varepsilon} (\tilde{x}-\tilde{a}+{\rm hot}),
\end{cases}
\end{equation}
where $\tilde a=1-a$,
\begin{equation}\label{eq-defa1}
a_1=\frac{J\tau^2}{2(1-J\tau)},
\end{equation}
and hot denotes higher order terms which have no influence on qualitative and low order quantitative features of canard explosion \cite{krupa-szmolyan:01}.
If the hot terms are omitted \eqref{eq-syscm} differs from the classical van der Pol system by the term $a_1\eps\tilde x$. Note that $a_1$ is positive
and blows up as $\tau$ approaches $1/J$.
The coefficient $a_1$ does not influence the coefficient $A$ defined in \cite{krupa2001extending}, which determines the criticality of canard explosion.
The feature changed by $a_1$ is the position of the Hopf and canard curves in the $(a, \eps)$ plane. For $\tau=0$ the Hopf curve is 
is given by $a=1$ and the canard curve is in the half plane $a<1$. As $\tau$ increases, the two curves turn to the right and eventually are both located
in the $a>1$ half plane. This feature allows for a very interesting version of a canard explosion: starting with $a<1$ and $\tau=0$ one can follow the 
evolution of the stable limit cycle as $\tau$ is increases while $a$ is kept fixed. Due to the movement of the canard line the parameter point $(a,\tau)$
approaches and eventually passes through the canard line, which gives a canard explosion. This is shown in Figure \ref{fig:fullCanard}.

\subsection{Existence of connections}
In order to complete our proof of the presence of canards explosions, we now investigate the persistence of connections, in the fast system, from the saddle fixed point to one branch of the critical manifold (hypothesis {\bf (H3)}). To this end, one needs to show a global convergence result for a one-dimensional delayed differential equation, which is a complex problem. Ample numerical simulations (see Fig.~\ref{fig:Connections}) show that such connections exist when $J=2$ and $0<\tau <1/J$. 
\begin{figure}[htbp]
	\centering
		\subfigure[$y=0$]{\includegraphics[width=.45\textwidth]{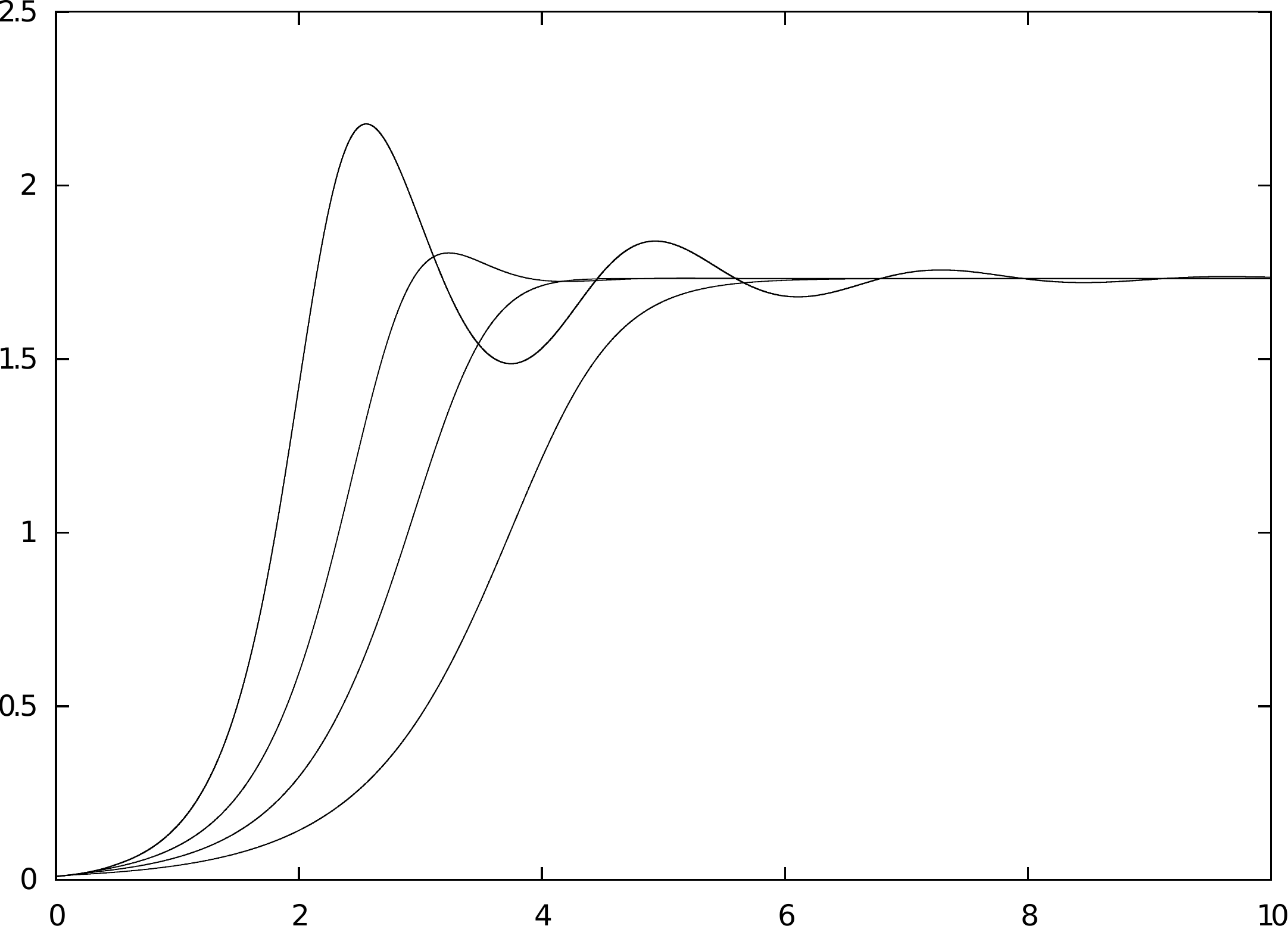}}
		\subfigure[$\tau=0.4$]{\includegraphics[width=.45\textwidth]{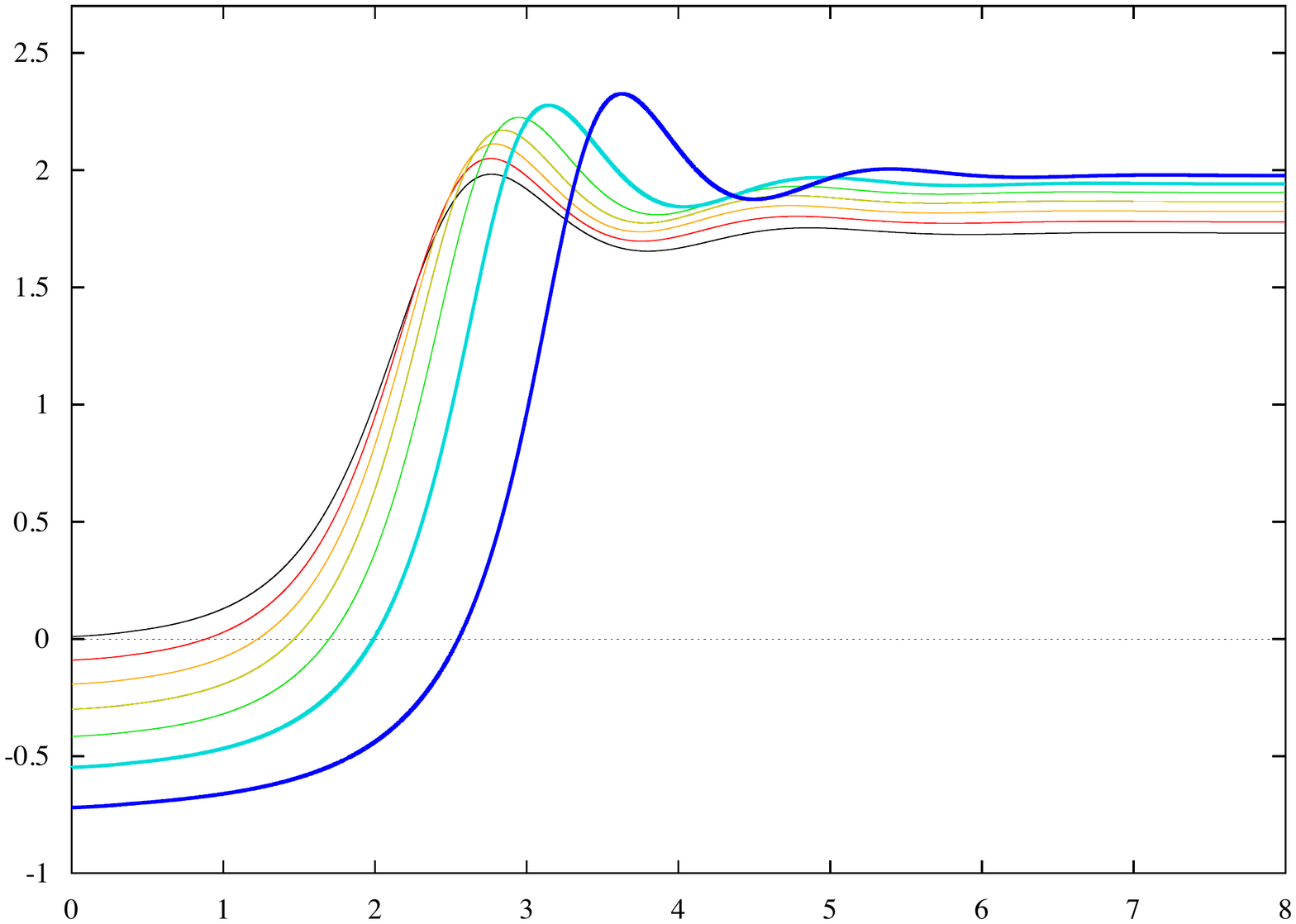}}
	\caption{Persistence of connections for $J\,\tau<1$. Here, $J=2$ fixed. (a) $y=0$ and different values of $\tau$ (from right to left, $\tau=0.1,\,0.2,\,0.3,\,0.4$). (b) $\tau=0.3$ fixed, and $y$ ranging from $0$ to $0.6$ (steps $0.1$). Initial condition was always set to $x_t=x_0$ for $t\in [-\tau,0]$ and $x(0)=x_0+0.01$ (where $x_0$ is the saddle fixed point). We observe in (a) that the connection persists for increasing values of $\tau$, but the attractivity of the fixed point decays, and in (b) that the connections persist all along the branch of unstable fixed points.}
	\label{fig:Connections}
\end{figure}
% 
% Existence of connections for $\vert y\vert < 2/3$ such that the absolute value is smaller are much more complex to handle, and indeed, in cases where Hopf bifurcations may arise, the stability of the fixed point on the stable manifold does not imply the existence of connections. For instance, considering $\vert y\vert <2/3$ and parameters such that the point $x_+(y)$ is stable for the delayed problem but for parameters such that $J\tau>1$. We have seen that a family of unstable limit cycles arising from the subcritical Hopf bifurcation may circle the fixed point, possibly preventing the existence of connections.
% `'
While we conjecture that connections persist for any $J\tau<1$, we demonstrate:
\begin{prop}\label{pro:Connections}
	Connections from the saddle fixed point to one of the stable fixed points exist if $J\tau<\frac 1 2 $.
\end{prop}

\begin{proof}
	We now consider $\vert y \vert < 2/3$, and denote $x_-(y)<x_0(y)<x_+(y)$ the three fixed points of the system. The fixed point $x_0(y)$ was shown to be unstable, while the other two fixed points are linearly stable. We aim at showing that solutions starting on the unstable manifold of the saddle solution $x(t)\equiv x_0(y)$ leave the neighborhood of this solution and converge towards the equilibrium $x(t)\equiv x_+(y)$ as time evolves (the system being symmetric by the transformation $(x,y)\mapsto (-x,-y)$, this case shows existence of connections on both sides of the unstable equilibrium). We define $z=x-x_+(y)$ and introduce the van der Pol potential centered around this point:
	\[V(z)=-\int_0^z \psi(z)\,dz\]
	where $\psi$ is the flow of van der Pol system:
	\[\psi(z) =-\bigg( (-1+x_+^2(y)) z + x_+(y)z^2 +\frac{z^3}{3}\bigg) =: -\bigg(\alpha z + \beta z^2 + \frac{z^3}{3}\bigg).\]
	Let $z(t)$ a solution of the delayed van der Pol system:
	\begin{equation}\label{eq-vdptrans}
	\dot z(t) = \psi(z(t)) + J \Big(z(t)-z(t-\tau)\Big).
	\end{equation}
		\\[1ex]
	We assume $J\tau<1/2$ and define $\rho=J\tau/(1-J\tau)$.
	We gather a few facts on $\psi$:
	\[
	\psi(z)=-((\beta^2-1)z+\beta z^2+\frac{1}{3}z^3,\quad \beta\in (1, 2],
	\]
	and
	\[
	\psi'(z)=-(\beta^2-1)+2\beta z+z^2.
	\]
	Let $z_{max}$ be the local maximum point of $\psi$.
	One can verify by direct computation that
	\begin{equation}\label{eq-zpsimax}
	z_{max}=1-\beta,\quad \psi_{max}=\frac{1}{3}(\beta-1)^2(\beta+2).
	\end{equation}
      We now prove the existence of a trapping region. We begin with the following estimate:
	\begin{align}\label{eq-estflip}
	\begin{split}
	z(t)-z(t-\tau)=&\int_{t-\tau}^tz'(\sigma) d\sigma\\
	                       &=\int_{t-\tau}^t\psi(z(\sigma)d\sigma +J\int_{t-\tau}^t z(\sigma)-z(\sigma-\tau)d\sigma\\
	                       &\le\tau\psi_{\rm max}+J\tau \max_{\sigma\in (t-\tau, t)}  (z(\sigma)-z(\sigma-\tau)).
       \end{split}
	\end{align}
	If $z(t)$ is defined on $(-\infty,\infty)$ then
	\begin{equation*}
	\max_{\sigma\in (-\infty, t)} z(\sigma)-z(\sigma-\tau)\le\tau\psi_{\rm max}+J\tau \max_{\sigma\in (-\infty, t)}  (z(\sigma)-z(\sigma-\tau)).
	\end{equation*}
	It now follows that
	\begin{equation*}
	\max_{\sigma\in (-\infty, t)}J( z(\sigma)-z(\sigma-\tau))\le\psi_{\rm max}\rho<\psi_{max}.
	\end{equation*}
	This implies that
	\[
	\min_{\sigma\in (-\infty, t)} \psi(z(\sigma))\ge-\psi_{\rm max}\rho
	\]
	since $\psi(z(\sigma))=-\psi_{\rm max}\rho$ implies $z'(\sigma)<0$. Hence
	\[
	\max_{\sigma\in (-\infty, t)} |\psi(z(\sigma))|\le\psi_{\rm max}
	\]
	Now we can adapt the earlier argument to estimate $|z(t)-z(t-\tau)|$:
	\begin{align}\label{eq-estflip1}
	\begin{split}
	|z(t)-z(t-\tau)|=&\left |\int_{t-\tau}^tz'(\sigma) d\sigma\right |\\
	                       &\le\left |\int_{t-\tau}^t\psi(z(\sigma)d\sigma\right | +J\left |\int_{t-\tau}^t z(\sigma)-z(\sigma-\tau)d\sigma\right |\\
	                       &\le\tau\psi_{\rm max}+J\tau \max_{\sigma\in (t-\tau, t)}  (|z(\sigma)-z(\sigma-\tau)|).
	\end{split}
	\end{align}
	Continuing as above we obtain
	         \begin{equation*}
	\max_{\sigma\in (-\infty, t)} J |z(\sigma)-z(\sigma-\tau)|<\psi_{\rm max}\rho.
	\end{equation*}
	We now have
	\begin{equation}\label{eq-estrap0}
		\der{}{t}V(z(t)) = -\psi^2(z(t)) + J \psi(z(t))(z(t)-z(t-\tau))<-\psi(z(t))^2+\psi_{\rm max}^2\rho^2.
	\end{equation}
	It follows that the solution must enter the region given by 
	\begin{equation}\label{eq-deftrap}
	\{z\in (z_{\rm max}, -z_{\rm max})\; :\; |\psi(z)|< \psi_{max}\rho\}.
	\end{equation}
	\begin{rem}\label{rem-ref}
	Note that there exists a constant $K>0$ such that, if $z\in(z_{max},-z_{max})$ satisfies $|\psi(z)|<\rho\psi_{max}$ then $|\psi'(z)|>1/K$.
	\end{rem}
	Let $m>0$ be an integer large enough so that $\tilde\rho=\rho+(J\tau)^mK<1$. We now assume that $t$ is sufficiently large so that
	$z(\sigma)$ satisfies \eqref{eq-deftrap} for $\sigma\in [t-m\tau, \infty)$. We now refine the estimate \eqref{eq-estflip1}, applying it iteratively $m$ times,
	to get
	\begin{align}\label{eq-estflip2}
	\begin{split}
	J|z(t)-z(t-\tau)|=&J\left |\int_{t-\tau}^tz'(\sigma) d\sigma\right |\\
	                       &\le J\left |\int_{t-\tau}^t\psi(z(\sigma)d\sigma\right | +J^2\left |\int_{t-\tau}^t z(\sigma)-z(\sigma-\tau)d\sigma\right |\\
	                       &\le\sum_{l=1}^{m}(J\tau)^l\rho\psi_{\rm max}+(J\tau)^m \max_{\sigma\in (t-m\tau, t)}  (|z(\sigma)-z(\sigma-\tau)|).
	\end{split}
	\end{align}
	It follows that 
	\[
	J |z(t)-z(t-\tau)|<\psi_{\rm max}\rho\tilde\rho
	\]
	and consequently
	\begin{equation}\label{eq-estrap1}
		\der{}{t}V(z(t)) = -\psi^2(z(t)) + J \psi(z(t))(z(t)-z(t-\tau))<-\psi(z(t))^2+\psi_{\rm max}^2(\rho\tilde\rho)^2.
	\end{equation}
	for $t$ such that 
	$z(\sigma)$ satisfies \eqref{eq-deftrap} for $\sigma\in [t-m\tau, \infty)$. It follows that the solution must enter the region
	\begin{equation}\label{eq-deftrap1}
	\{z\in (z_{\rm max}, -z_{\rm max})\; :\; |\psi(z)|< \psi_{max}\rho\tilde\rho\}.
	\end{equation}
	Proceeding by induction we prove that for any integer $l>0$ there exists $t$ large enough so that the solution enters the region
	\begin{equation}\label{eq-deftrapl}
	\{z\in (z_{\rm max}, -z_{\rm max})\; :\; |\psi(z)|< \psi_{max}\rho(\tilde\rho)^l\}.
	\end{equation}
	It follows that $z(t)$ converges to $0$.\\[1ex]	% 
\end{proof}

\subsection{Stability hypothesis}
\noindent The last property we need to show in order to apply theorem~\ref{th-canoh} is the stability assumption {\bf(H4)}. To this end, we can readily express the contraction/expansion rates given by formula \eqref{eq-defR} thanks to the closed-form expression of the eigenvalues of the system in terms of the Lambert function $W_0$. These expressions were then evaluated numerically. These computations showed that
{\bf (H4)} holds for smaller values of $\tau$, but for $\tau>\tau_*\approx 0.354$ there exists a subinterval of $(-2/3, 2/3)$ where ${\bf (H4)}$ is violated,
see Fig. \ref{fig:Rates}.
\begin{figure}[h]
	\centering
		\subfigure[$\tau=0.3$  ]{\includegraphics[width=.32\textwidth]{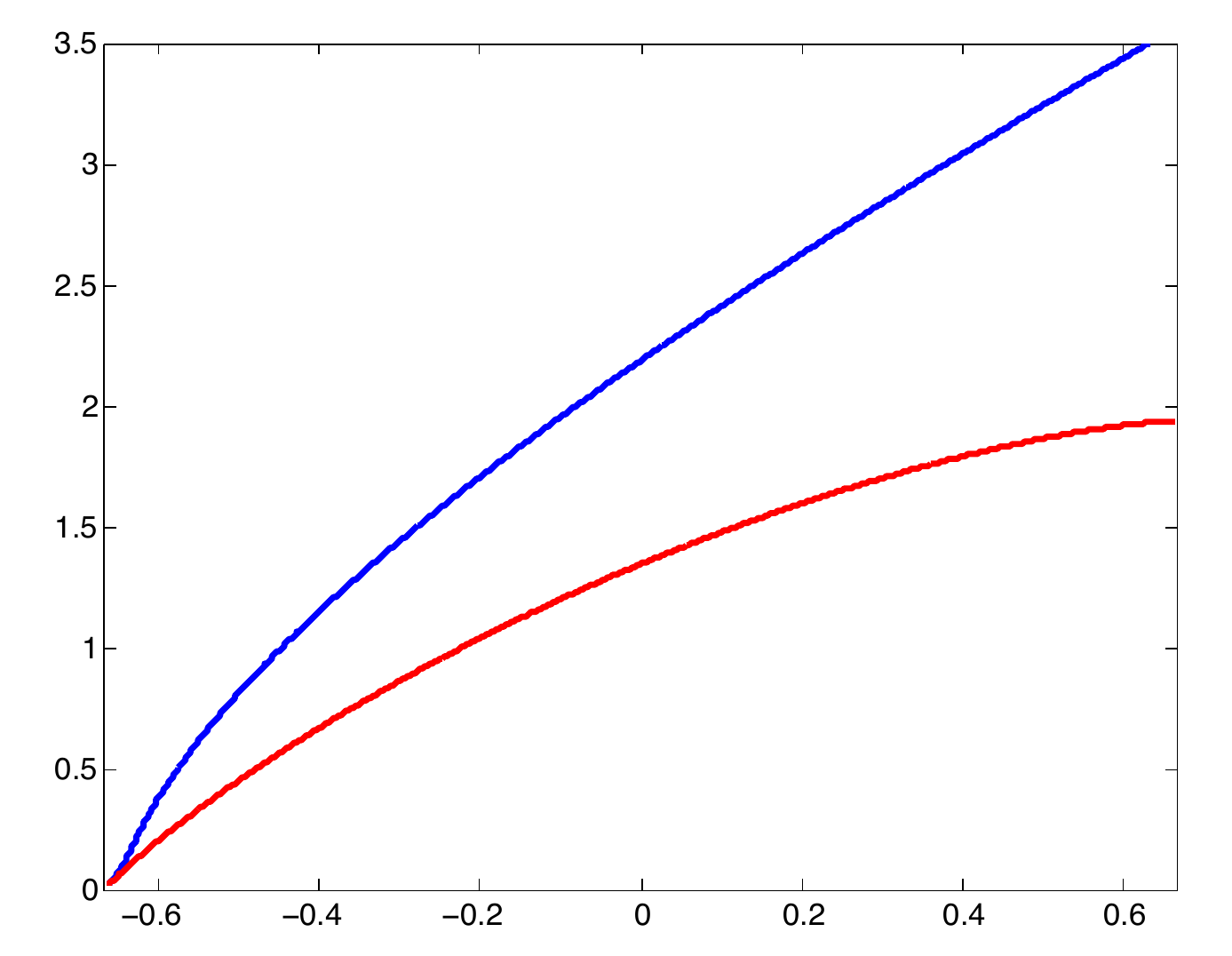}}
		\subfigure[$\tau=0.354$]{\includegraphics[width=.32\textwidth]{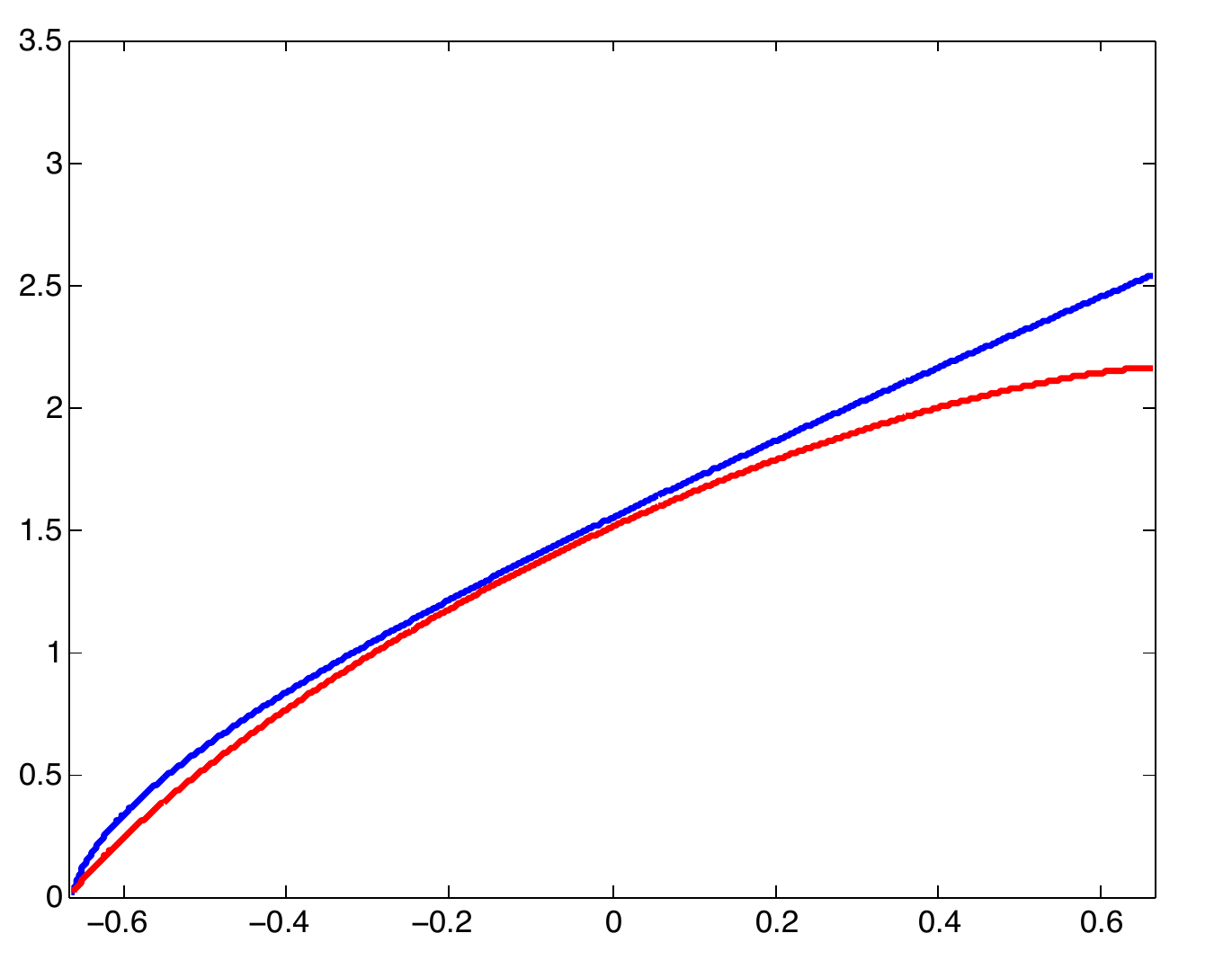}}
		\subfigure[$\tau=0.37$ ]{\includegraphics[width=.32\textwidth]{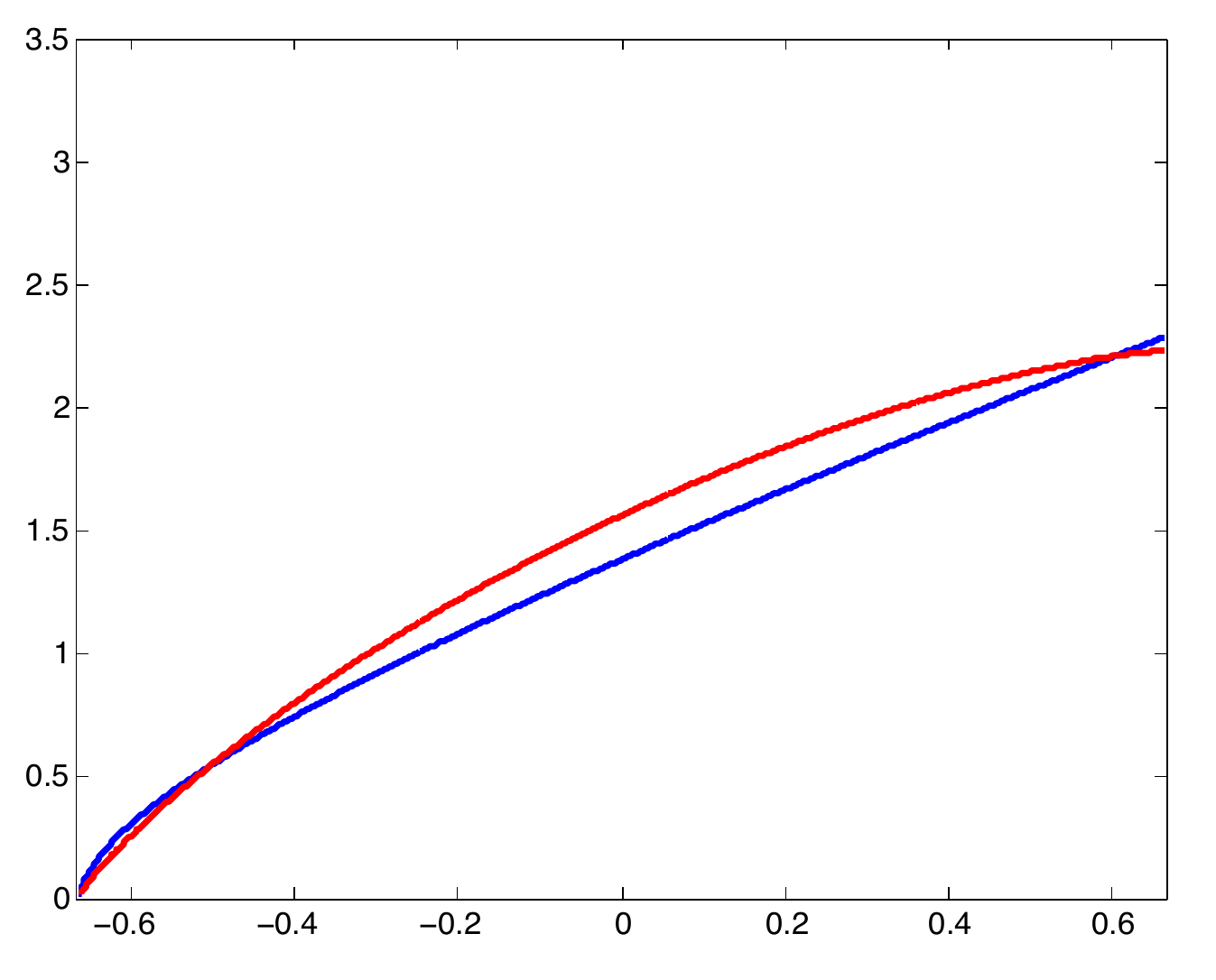}}
	\caption{Comparison between the rates $R_{n,+}(y_*)$ (blue curve) and $R_p(y_*)$ (red curve) for $J=2$ and three values of $\tau$. (a) {\bf (H4)} holds for $\tau=0.3$ (b) The limiting case, $\tau=0.354$, where $R_{n,+}(y_*)=R_p(y_*)$
	appear to be equal at one point. (c) {\bf (H4)} no longer holds for $\tau = 0.37$.}
	\label{fig:Rates}
\end{figure}

It follows from Theorem \ref{th-canoh} that for $\tau<\tau_*$ there exists a canard explosion analogous to the one occurring in the classical van de Pol
system. However, for $\tau>\tau_*$ this version of canard explosion may no longer be present. There is still a weak version of canard explosion. Indeed, canards with `large head' must exist and be stable. Moreover, canard explosion occurs and is locally supercritical (as shown above). All canard cycles 
are exponentially close in the parameter space to a segment of the canard solution
and therefore the locus of the canard cycle is exponentially close to that of the maximal canard (see Remark \ref{rem-width}). 
It is very difficult to determine numerically whether complex dynamics occurs  during a canard transition. Partial verification 
is provided by period doubling cascades of small amplitude cycles, found for $\tau\approx 0.4$ (see Fig~\ref{fig:PeriodDoublings}). Such period doubling cascades can no longer
be observed as $\eps\to 0$, but it is possible that they turn into period doubling cascades of canards when $\eps$ is  sufficiently small.
\begin{figure}[h]
	\centering
		\subfigure[Cycles]{\includegraphics[width=.48\textwidth]{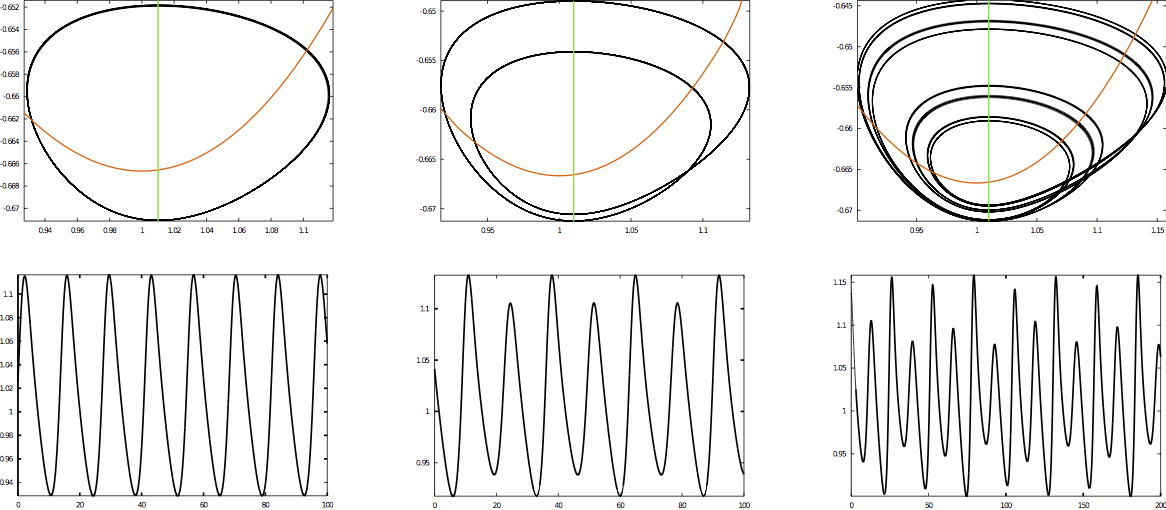}}
		\subfigure[Chaotic Orbit]{\includegraphics[width=.48\textwidth]{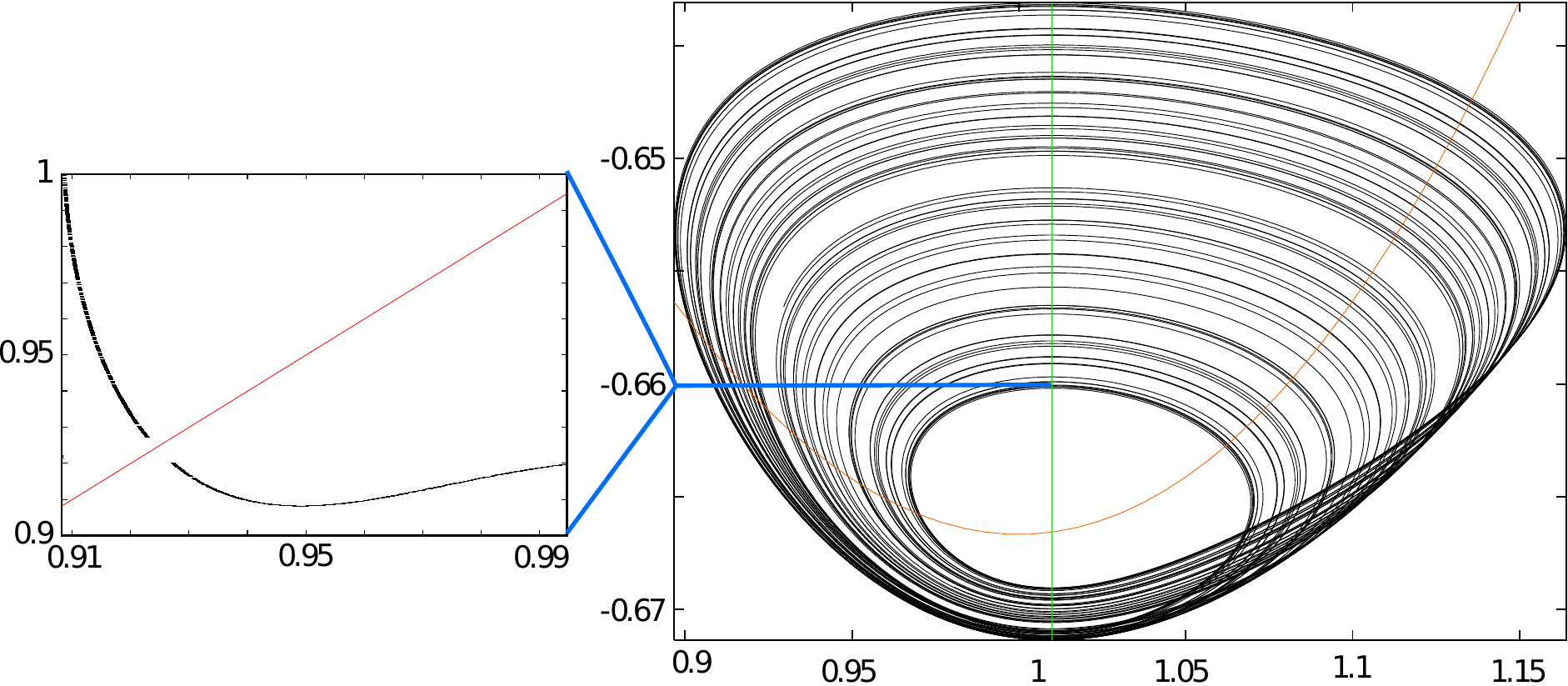}\label{fig:ChaosPeriodDoubling}}
			\caption{Simulation of the delayed van der Pol equation for $a=1.01$, and $\varepsilon$ was increased to $0.05$ in order to follow the phenomenon.  (a): $\tau=0.4$, $\tau=0.401$ and $\tau=0.408$. The cycle arising from the Hopf bifurcation, after a few period doubling bifurcations, shows a chaotic profile $\tau=0.41$ (b), as illustrated by the Ruelle plot (left) of the permanent dynamics on a Poincar\'e section (blue line at $w=-0.66$).}
	\label{fig:PeriodDoublings}
\end{figure}

The parameter point $(a, \tau)=(1, 1/J)$, it is shown~\cite{krupa-touboul:14b} that the system undergoes a generic subcritical Bogdanov-Takens (BT) bifurcation. The proof proceeds by center reduction manifold. It is however easy to check that this point has the typical BT degeneracy, noting that the two-dimensional vector space of affine functions belongs to the nullspace of the linearized operator at this point $\mathcal{L} u =J(u_t-u_{t-\tau})$, and moreover one notes that the linearized operator maps $t\mapsto \beta t$ on the constant function $\beta$, ensuring that the linearized equation has the typical shape of the BT singularity. As a result, for $\tau>1/J$, there is a Hopf bifurcation of the fast system on $S_+$, so that \textbf{(H2)} no longer holds for $J\tau \geq 1$.

\begin{figure}
	\centering
		\includegraphics[width=.5\textwidth]{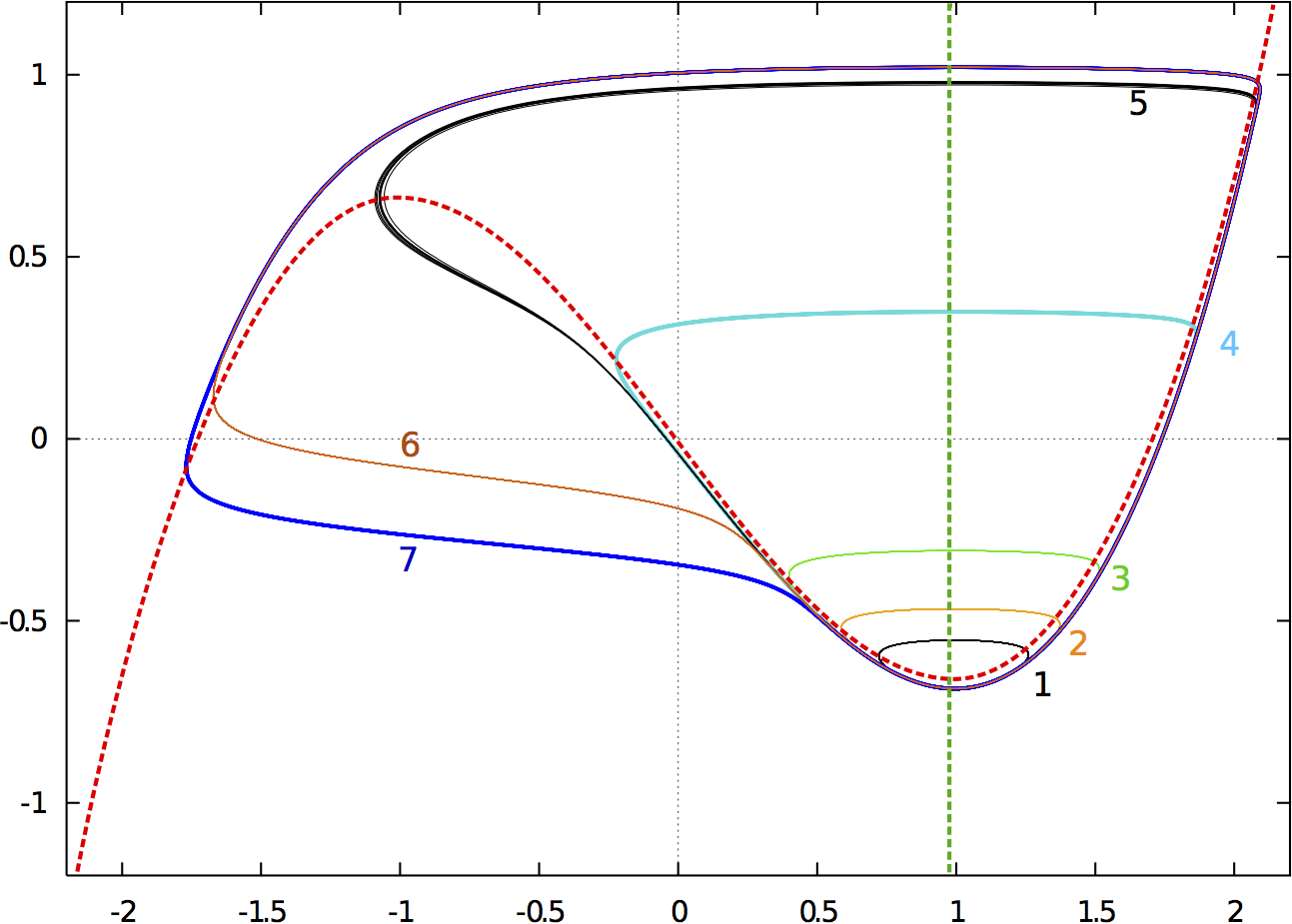}
	\caption{Delay-induced canards in the delayed van der Pol system. $\varepsilon=0.05$, $a=0.995$. Dotted lines correspond to the nullclines (red: $x$ nullcline, green: $y$ nullcline) and the curves represent trajectories in the phase plane $(x,y)$ for different values of the delay. 1. $\tau=0.01$, 2. $\tau=0.07$, 3. $\tau=0.085$, 4.$\tau=0.08951569008$, 5.$\tau=0.08951569009$, 6. $\tau=0.089516$, 7. $\tau=0.0896$.}
	\label{fig:fullCanard}
\end{figure}
\section{Discussion}
In this paper, we have extended the theory of canard explosions to a class of delayed differential equations, and applied the theory to a delayed version of the van der Pol oscillator. We proved that for a wide range of delays, the canard explosion is similar to the case of a system with one slow and one fast variables. However, as delays are increased, several phenomena may occur. We have shown that delays can induce a destabilization the family of small canard cycles. Moreover, delays may destabilize the stable branches of the critical manifold preventing global canard explosion from occurring. For instance in the delayed vdP system, such a phenomenon arises as delays are sufficiently large, through a Bogdanov-Takens bifurcation. In that regime, complex oscillatory patterns arise, and will be further analyzed in~\cite{krupa-touboul:14b}. Our work therefore extends the results demonstrated in~\cite{campbell-etal:09,stone2004stability} to equations with non-small delays or distributed delays. 

This study opens the way to the analysis of canard phenomena in delay equations. In particular, we expect to find interesting dynamical phenomena in systems with more slow directions and delays, where small subthreshold oscillations due to folded singularities interact with the oscillatory instability induced by the delay. In such systems, delay-induced MMOs or bursting shall emerge, mediated by the presence of canard solutions. 

We eventually note that such intricate oscillatory patterns may also emerge from instabilities in the fast equation with just one slow variable. This is in particular the case of the delayed vdP system which was shown to undergo canard explosion in section~\ref{sec:CanardvdP}. Indeed, as delays are increased to $\tau>1/J$, the equilibria of the fast equation loose stability, leading the system to regimes of complex dynamics including Mixed Mode Oscillations, bursting and chaos. We show such solutions in Fig.~\ref{fig:ComplexDynamics}. These are further investigated in~\cite{krupa-touboul:14b}. 

\begin{figure}[htbp]
	\centering
		\subfigure[$\tau=0.4$: small oscillations]{\includegraphics[width=.3\textwidth]{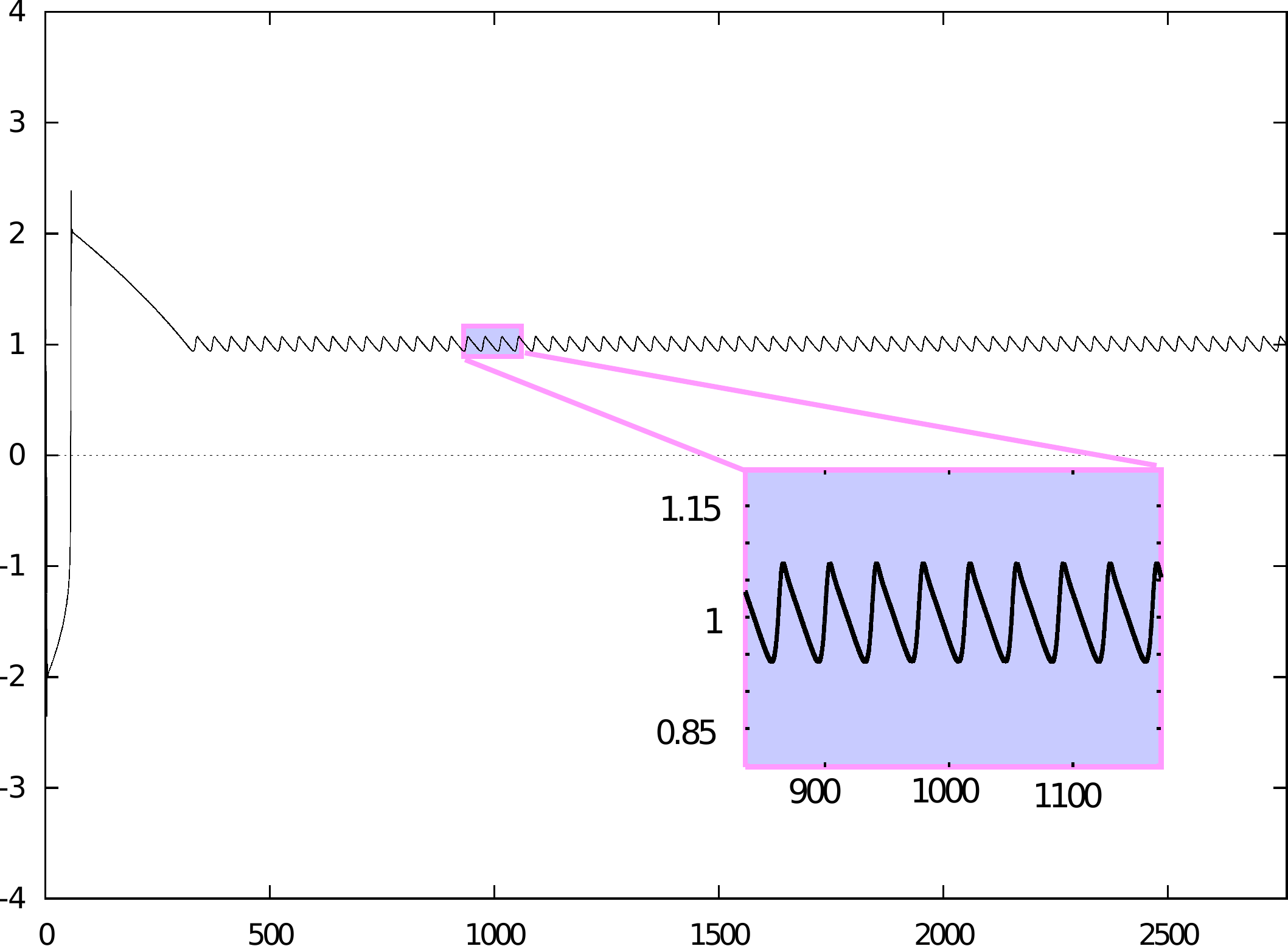}}
		\subfigure[$\tau=0.45$: MMOs]{\includegraphics[width=.3\textwidth]{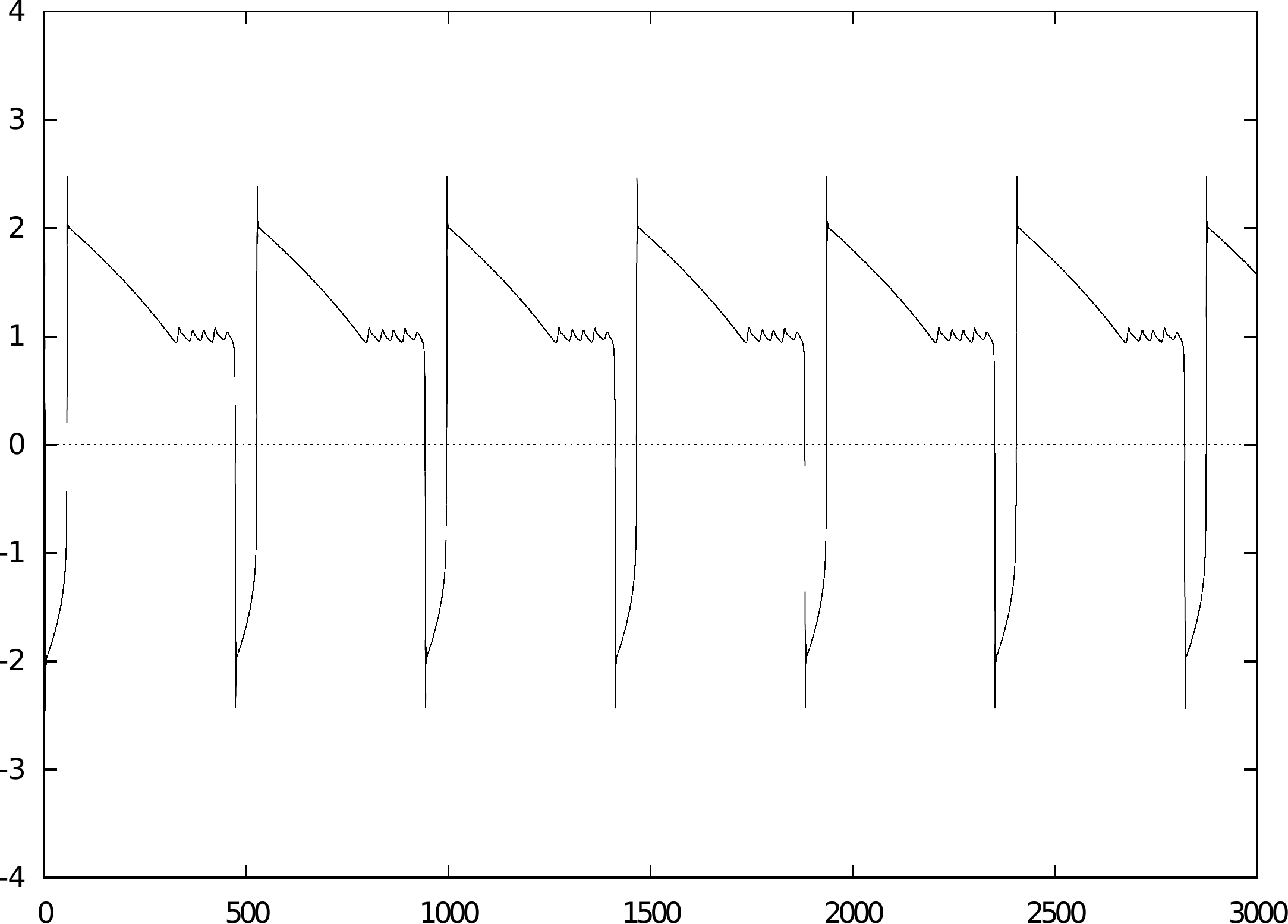}}
		\subfigure[$\tau=1$: Bursts]{\includegraphics[width=.3\textwidth]{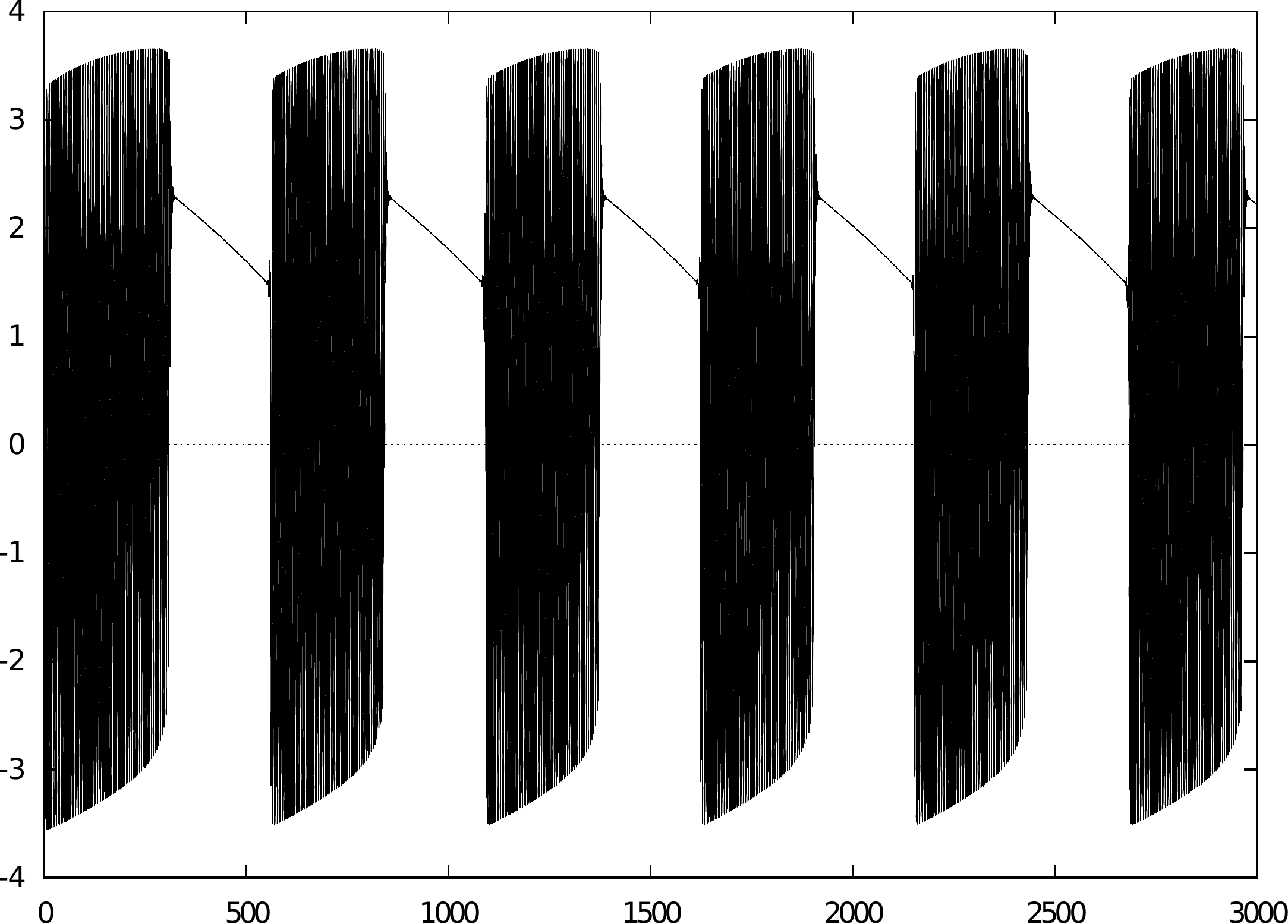}}
	\caption{Simulations of the delayed vdP system for $a=1$, $J=2$ and different values of $\tau$. After the canard explosion, complex oscillations arise, including Mixed-Mode oscillations and bursting. }
	\label{fig:ComplexDynamics}
\end{figure}
\appendix
\section{Fenichel theory for delayed differential equations}\label{append:fenichel}
We present here a sketch of the proof of theorem~\ref{thm-fen}. For simplicity we assume that $g(0,y,0)\neq 0$ for $y_0\le y\le y_1$, and without loss of generality assume $g(0,y,0)> 0$
(this is the only case we need in our application). 
By modifying the slow flow we can make sure that there
exist equilibrium points 
$y_{e0}<y_0$ and $y_{1e}>y_1$
such that  the hypothesis (H1) still holds, the slow flow is unchanged
on $[y_0, y_1]$ and there is a solution $\xi(t)$
of 
\begin{equation}\label{eq-slow}
\dot y=g(0,y,\eps).
\end{equation}
The linear problem for \eqref{eq-delsf} is as follows:
\begin{align}\label{eq-delsflin}
\begin{split}
x'=& \int_0^h d\zeta(y, \tau) x(t-\tau)d\tau,\\
y'=&\eps g(0, y).
\end{split}
\end{align}
Using the solution $\xi(t)$ we get the following equivalent formulation of \eqref{eq-delsflin}
\begin{equation}\label{eq-lin1}
v'= \int_0^h d\zeta(\xi(\tau), \tau) v(t-\tau)d\tau
\end{equation}
and let $T(t)$ be the solution operator.
Let 
\begin{align}\label{eq-xpxm}
\begin{split}
&X^-=\{v\in BC_\eta(\R; X))\; :\; \lim_{t\to \infty} e^{-\eta t} \|T(t)v\| = 0\},\quad X^-(\tau)=\{v(\tau)\; :\; v\in X^-\}\\
&X^+=\{v\in BC_\eta(\R; X))\; :\; \lim_{t\to -\infty} e^{\eta t} \|T(t)v\| = 0\},\quad X^+(\tau)=\{v(\tau)\; :\; v\in X^+\}.
\end{split}
\end{align}
Let $P_-(\tau)$ qnd $P_+(\tau)$ be projections onto the spaces $X^-(\tau)$ and $X^+(\tau)$ 
with kernel $X^+(\tau)$ and $X^-(\tau)$, respectively.  Consider 
\begin{align}\label{eq-delsflintilda}
\begin{split}
x'=& \int_0^h d\zeta(y, \tau,\eps) x(t-\tau)d\tau+\eps\tilde{f}(y),\\
y'=&\eps g(0, y)
\end{split}
\end{align}
where $\tilde{f}$ is defined by $f(0,y,\eps)=f(0,y,0)+\eps\tilde{f}(y)$. We also define the operator ${\mathbf r}$ which assigns to a an element of $X$ its value at $0$.
Further, we modify ${\mathbf r}$, replacing it with ${\mathbf r}_{\rm mod}$, to ensure that~\eqref{eq-delsf} is equal to~\eqref{eq-delsflintilda}
outside a small neighborhood of $S_0=\{(0,y),\; y\in [y_{0e}, y_{1e}]\}$,
see \cite{diekmann1995delay} for details of a similar construction in the context of the center manifold theorem.
Finally, we define the operator ${\mathcal F}=({\mathcal F}_f,{\mathcal F}_s)$ which
maps the space $BC_\eta(\R; X\times\R)$ into $BC_\eta(\R; X\times\R)$ ($\eta$ has to be chosen within the
spectral gap), given by
\begin{align}\label{eq-opsf}
\begin{split}
{\mathcal F}_f(\psi(t),\phi(t)))=
&T(t)\psi(0) +\int_{-\infty}^t T(t-\tau)P_-(\tau){\mathbf r}_{\rm mod}(f(\psi(\tau),\phi(\tau),\eps))d\tau\\
&+\int_{\infty}^t T(t-\tau)P_+(\tau) {\mathbf r}_{\rm mod}(f(\psi(\tau),\phi(\tau),\eps))d\tau\\
{\mathcal F}_s(\psi(t),\phi(t))=&\phi(0)+\eps\int_0^t {\mathbf r}_{\rm mod}(g(\psi(\tau),\phi(\tau),\eps))d\tau.
\end{split}
\end{align}
For each choice of ${\mathbf r}_{\rm mod}$ there exists a unique fixed point of ${\mathcal F}$ corresponding
to a solution of \eqref{eq-delsf} which defines a slow manifold.

To construct $W^u(S_\eps)$ we use a different version of the operator ${\mathcal F}$,
now defined on  $BC_\eta(\R_-; X\times\R)$ into $BC_\eta(\R_-; X\times\R)$, where $\R_-$
are non-positive real numbers. First of all, 
we note that the space $X_+$ is finite, see Chapter IV in \cite{diekmann1995delay}.
Let $k={\rm dim}\, X_+$.
The fixed point equation
\begin{equation}\label{eq-fp}
(\psi,\phi)={\mathcal F}(\psi(t),\phi(t))
\end{equation}
defines a $k$ dimensional submanifold of $BC_\eta(\R_-; X\times\R)$, for any $\eta>0$ within the spectral gap. 
The elements of this manifold define solutions of \eqref{eq-delsf} on $\R$.
The manifold $W^u(S_\eps)$ is obtained as the union of these solutions.
Smoothness of $S_{\eps}$ or $W^u(S_\eps)$ is proved analogously as in~\cite{diekmann1995delay}. 

\section{Stone-Campbell small delay expansion}\label{append:Campbell}
In the whole manuscript, we have worked with arbitrary delays. General analysis was provided for canard explosions, and an analysis of the delayed van der Pol equation was provided and showed a vast repertoire of behaviors above $\tau=1/J$. In~\cite{stone2004stability}, Stone, Campbell and Erneux proposed a method in order to characterize canard explosions for delayed equations in the limit of small delays. In that regime, one may use the perturbation result of Chicone~\cite{chicone:03} showing, in our case of the delayed vdP system, that the system has a two-dimensional inertial manifold for $\tau$ sufficiently small. On this manifold, and in the limit of small delays, the term $x(t)-x(t-\tau)$ is well approximated by $\tau x'(t) -\tau^2/2 x''(t)+O(\tau^3)$. The method consists in using the fact that equation~\eqref{eq:vdpdelay} can be approximated at first order in $\tau$ by the solutions of the following ordinary differential equation:
\[\begin{cases}
	(1-J\tau)x' = x-\frac{x^3}{3} +y\\
	y' = \varepsilon (a-x)
\end{cases}\]
which can be written, through the change of time $\theta=t/(1-J\tau)$:
\begin{equation}\label{eq:SmallDelay}
\begin{cases}
		\der{x}{\theta} = x-\frac{x^3}{3} +y\\
		\der{y}{\theta} = \varepsilon(1-J\tau) (a-x)
	\end{cases}	
\end{equation}
which precisely corresponds to the non-delayed van der Pol equation with a modified slow timescale $\tilde{\varepsilon}(\tau)=\varepsilon(1-J\tau)$. Classical methods from the ODE domain can thus be applied in order to show that a canard explosion occurs, and to provide an approximate formula for the canard point. This follows classical theory that we review here in the context of the vdP equation. Under a few geometrical conditions~\cite{krupa-szmolyan:01}, a slow-fast dynamical system generically presents canard explosion. These results can be summarized as follows. We consider a two-dimensional slow-fast system of type:
\[
\begin{cases}
	\varepsilon \dot{x}=f(x,y,\lambda,\varepsilon)\\
	\dot{y}=g(x,y,\varepsilon)
\end{cases}
\]
where $\lambda\in(-\lambda_0,\lambda_0)$ is a parameter, and make the following assumptions:
\renewcommand{\theenumi}{A\arabic{enumi}}
\begin{enumerate}
	\item The critical manifold $\Sigma=\{x,y: f(x,y,\lambda,0)=0\}$ is S-shaped for all $\lambda$, i.e. can be written as $y=\varphi_{\lambda}(x)$, and $\varphi_{\lambda}$ has exactly one non-degenerate minimum $x_l(\lambda)$ and maximum $x_r(\lambda)$.
	\item The submanifolds $S_l=\Sigma \cap \{x< x_l\}$ and $S_l=\Sigma \cap \{x> x_r\}$ are attracting ($\partial f/\partial x <0$) and $S_m=\Sigma\cap\{x_l<x<x_r\}$ is repulsive ($\partial f/\partial x >0$) for the layer problem
	\item Both folds are generic for $\lambda\neq 0$, i.e. for $x^*=x_l$ or $x_r,$
	\[\frac{\partial^2 f}{\partial x^2}(x^*, \varphi(x^*),\lambda,0)\neq 0 \qquad \frac{\partial f}{\partial x}(x^*, \varphi(x^*),\lambda,0)\neq 0 \qquad g(x^*,\varphi(x^*),\lambda,0)\neq 0\]
	and for $\lambda=0$, one of the folds is a non-degenerate canard point, i.e. satisfies the two first differential conditions of the fold and 
	\[\frac{\partial g}{\partial x}(x^*, \varphi(x^*),0,0)\neq 0 \qquad \frac{\partial g}{\partial \lambda}(x^*, \varphi(x^*),0,0)\neq 0\]
	\item When $\lambda=0$, the slow flow on $\Sigma$, namely $g(x,\varphi_0(x),0,0)/\varphi_0'(x)$, is strictly positive on $S_l\cup S_m\cup \{x_r\}$ and strictly negative on $S_r$. 
\end{enumerate}
Then for $\varepsilon$ and $\lambda$ sufficiently small, the system has a unique equilibrium converging to the canard point as $(\varepsilon,\lambda)\to 0$, and this point looses stability as $\lambda$ is increased through a Hopf bifurcation. The small cycles arising from the Hopf bifurcation (canard cycles) joins relaxation oscillations within an exponentially small interval of $\lambda$ of order $O(e^{-K/\varepsilon})$.

It is very easy to see that these conditions readily apply to the case of the non-delayed van der Pol equation, implying the existence of a canard explosion as a function of the parameter $a$. This is also true of the small delay ODE~\eqref{eq:SmallDelay}. Indeed, $\tilde{\varepsilon}(0)=\varepsilon$, and for $J>0$ as assumed here, the function $\tau\mapsto \varepsilon(\tau)$ is non-increasing. This type of variation of the parameter is not usual in the analysis of the van der Pol equation, since our delay parameter acts precisely on the timescale of the slow variable. In order to follow blow-up method used in~\cite{krupa2001extending,krupa-szmolyan:01}, we shall define $\tilde{x}=-(x-1)$ and $\tilde{y}=(y-2/3)$. These variables satisfy the equations:
\[\begin{cases}
	\der{\tilde{x}}{\theta} = \tilde{x}^2-\frac{\tilde{x}^3}{3} -\tilde{y}\\
	\der{\tilde{y}}{\theta} = \tilde{\varepsilon} (\tilde{x}-\tilde{a})
\end{cases}\]
with $\tilde{a}=1-a$. It is then trivial to reduce it to canonical form~\cite[Section 3.1]{krupa2001extending}:
\[\begin{cases}
	\der{\tilde{x}}{\theta} = -\tilde{y} h_1 + \tilde{x}^2 h_2(\tilde{x}) + \varepsilon h_3\\
	\der{\tilde{y}}{\theta} = \tilde{\varepsilon}(x\,h_4 -\lambda h_5+y h_6)
\end{cases}\]
with $h_1=h_4=-h_5=1$, $h_2=1-x/3$ and $h_3=0$, compute the coefficients 
\[a_1=\derpart{h_3}{x}=0, \quad a_2=\derpart{h_1}{x}=0, \quad a_3=\derpart{h_2}{x}=-\frac 1 3, \quad a_4=\derpart{h_4}{x}=0, \quad a_5=h_6=0\]
and we therefore conclude that the Hopf bifurcation arises at $\tilde{a}_H(\varepsilon,\tau)=O(\tilde{\varepsilon}^2)$, this bifurcation is supercritical, and the maximal canard appears at 
\[\tilde{a}_c(\varepsilon,\tau)=\frac{\tilde{\varepsilon}}{ 8} + O(\tilde{\varepsilon}^2)\sim \frac{\varepsilon(1-J\tau)}{8 }.\]
Let us now consider the system with fixed $\tilde{a}>0$ small enough so that $\tilde{a}_H(\varepsilon, 0)<\tilde{a}<\tilde{a}_c(\varepsilon, 0)$. In that case, the system with $\tau=0$ present small canard oscillations. As $\tau$ is increased, the value of the effective parameter $\tilde{a}_c(\varepsilon,\tau)$ decreases. Increasing $\tau$ with fixed $\tilde{a}$ will hence induce a canard explosion. Simulations of the ODE system for small delay, provided in Fig.~\ref{fig:DelayInducedCanards}(a), indeed show that canard explosion as a function of the delays. This diagram is compared to simulations of the original delayed van der Pol system and the same qualitative scenario arises. However, for our choice of parameters, the delay corresponding to the bifurcation in the approximated ODE, close to $0.11$, is not very small and therefore it happens to slightly different from the value corresponding to the delay differential equation (close to $0.09$). Decreasing the value of $\varepsilon$ and taking $a$ closer to $1$ reduces the value of the $\tau$ corresponding to the bifurcation, making it closer from the simulations of the actual system\footnote{We chose the parameters for Fig.~\ref{fig:DelayInducedCanards} because it allows more flexibility for illustrating the phenomena (the canard explosion arises on a broader interval of values for $\tau$).}.

\begin{figure}
	\centering
		\includegraphics[width=.5\textwidth]{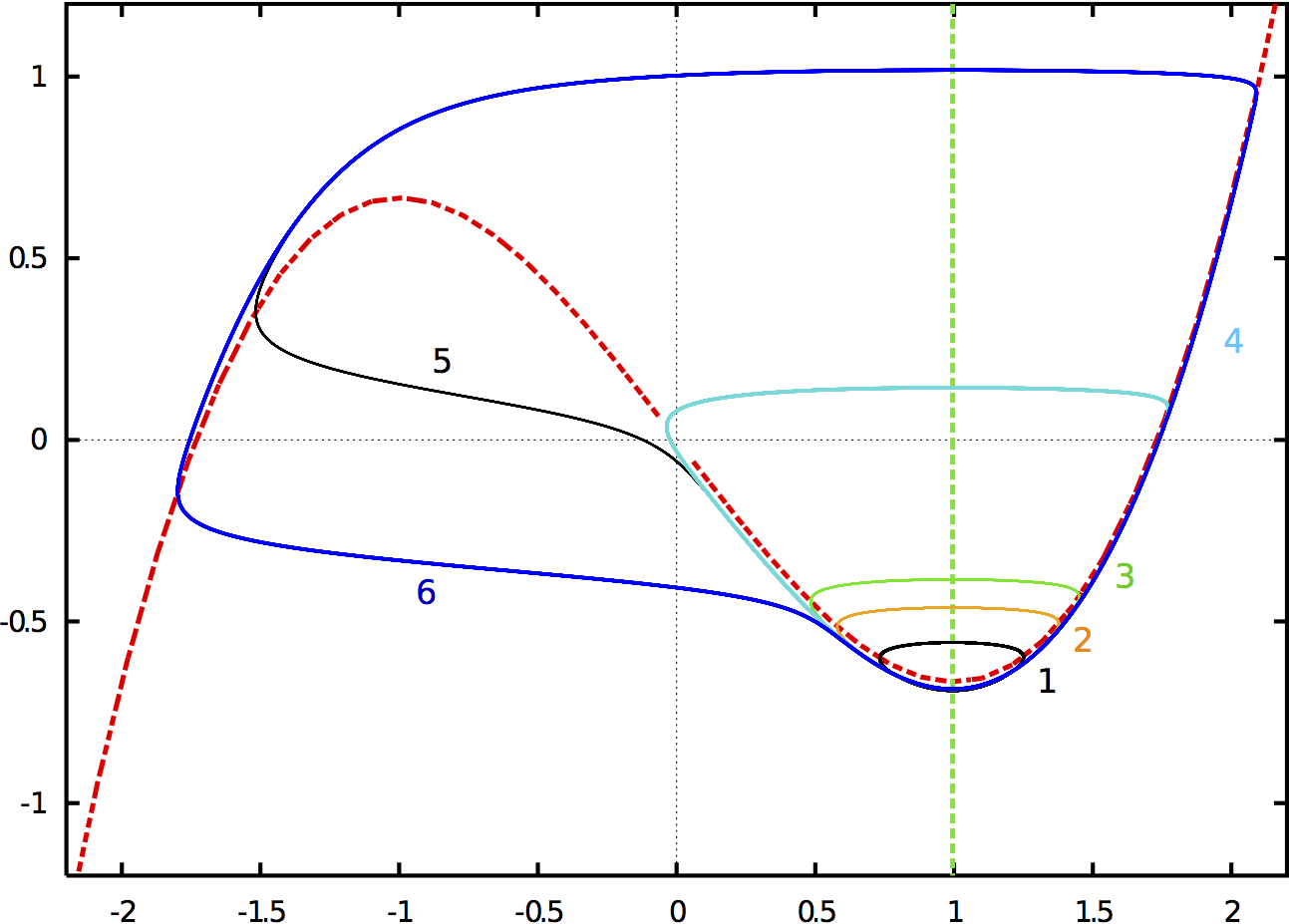}
	\caption{Delay-induced canards in small delay approximated ODE (compare with Fig.~\ref{fig:fullCanard}). $\varepsilon=0.05$, $a=0.995$. Dotted lines correspond to the nullclines (red: $x$ nullcline, green: $y$ nullcline) and the curves represent trajectories in the phase plane $(x,y)$ for different values of the delay. 1. $\tau=0.01$, 2. $\tau=0.11$, 3. $\tau=0.112$, 4.$\tau=0.112167578$, 5.$\tau=0.112167579$, 6. $\tau=0.1122$.}
	\label{fig:DelayInducedCanards}
\end{figure}

The simulations of the original delayed van der Pol equations~\eqref{eq:vdpdelay} for this set of parameters shows a very clear delay-induced canard explosion, as shown in Fig.~\ref{fig:DelayInducedCanards}: for values of $a$ smaller than $1$ and no delay, the system presents small oscillations corresponding to the presence of the Hopf bifurcation for $\tau=0$ and $a=1$. When increasing the value of the delay, the amplitude of this small cycle suddenly becomes very large, corresponding to relaxation oscillations. However, these cycles depart from the actual system and non-perturbative analysis in the delay is necessary in order to uncover these phenomena, as provided in the main text. This is even more true for $a>1$, in which case canard explosion arise for relatively large delays.

%\begin{acknowledgements}
%\end{acknowledgements}

% \bibliographystyle{spphys}       % APS-like style for physics

%% This is what we used:
% \bibliographystyle{spmpsci}      % mathematics and physical sciences
% \bibliography{../../canards}   % name your BibTeX data base

\end{document}